\documentclass{amsart}
\usepackage[twoside,lmargin=9em,rmargin=9em,top=4cm,bottom=4cm]{geometry}
\usepackage[dvipdfmx]{graphicx,xcolor}

\usepackage{amsmath}
\usepackage{amssymb}
\usepackage{amsthm}
\usepackage[all]{xy}

\newtheorem{thm}{Theorem}[section]
\newtheorem{prop}[thm]{Proposition}
\newtheorem{lemma}[thm]{Lemma}
\newtheorem{defi}[thm]{Definition}
\newtheorem{example}[thm]{Example}
\newtheorem{rem}[thm]{Remark}

\newtheorem{cor}[thm]{Corollary}

\def\QQ{\mathord{\mathbb{Q}}}

\def\ZZ{\mathord{\mathbb{Z}}}
\def\LL{\mathbb{L}}

\def\ket#1{\vert #1 \rangle}
\def\bra#1{\langle #1 \vert}
\def\kket#1{\vert #1 \rangle\!\rangle}
\def\bbra#1{\langle\!\langle #1 \vert}
\def\mket#1{\vert #1 )}
\def\mbra#1{( #1 \vert}
\def\inner#1#2{\left\langle #1|#2\right\rangle} 
\def\iinner#1#2{\langle\!\langle #1|#2\rangle\!\rangle}

\title[Generating functions of dual $K$-theoretic $P$- and $Q$-functions]{Generating functions of dual $K$-theoretic $P$- and $Q$-functions and boson-fermion correspondence}

\author[S. Iwao]{Shinsuke Iwao}
\address[S. Iwao]{Faculty of Business and Commerce, Keio University, Hiyosi 4-1-1, Kohoku-ku, Yokohama, Kanagawa 223-8521, Japan}
\email{iwao-s@keio.jp}

\begin{document}
\maketitle

\begin{abstract}
In this paper, we present a new algebraic description of Ikeda-Naruse's $K$-theoretic Schur $P$- and $Q$-functions and their dual functions in terms of neutral fermion operators. 
We introduce four families of ``$\beta$-deformed neutral-fermion operators'' depending on a parameter $\beta$, which reduce to the usual neutral-fermion operators when $\beta$ is zero.
Using these operators, we introduce two families of $\beta$-deformed vertex operators, power sums, and boson-fermion correspondences.
From commutation relations among these operators, we naturally derive the $K$-theoretic Cauchy kernel of Nakagawa-Naruse. 
Exploiting this fact, we show that the four $K$-theoretic functions can be realized as vacuum expectation values of certain $\beta$-deformed fermionic operators.
This presentation also allows us to derive generating functions for the dual $K$-theoretic $P$-and $Q$-functions, as conjectured by Nakagawa-Naruse.
\end{abstract}

\section{Introduction}

\subsection{Overview}\label{sec:Background}

The \textit{$K$-theoretic Schur $P$- and $Q$-functions}, introduced by Ikeda and Naruse~\cite{IKEDA201322}, are distinguished families of symmetric functions representing Schubert classes in the $K$-theory of maximal isotropic symplectic and orthogonal Grassmannians.  
The $K$-theoretic Schur $P$-function (abbreviated as the $K$-$P$ function) is typically denoted by \( GP_\lambda \), and the $K$-theoretic Schur $Q$-function (abbreviated as the $K$-$Q$ function) is denoted by \( GQ_\lambda \).  
These functions are indexed by \textit{strict partitions} \( \lambda = (\lambda_1 > \cdots > \lambda_r > 0) \), i.e., strictly decreasing sequences of positive integers.  
(They are also referred to as \textit{shifted stable Grothendieck polynomials}~\cite{lewis2024combinatorial}.)

For any finite \( n > 0 \), two polynomials \( GP_\lambda(x_1, \dots, x_n) \) and \( GQ_\lambda(x_1, \dots, x_n) \) are defined as symmetric polynomials in \( x_1, \dots, x_n \) with coefficients in \( \mathbb{Z}[\beta] \), where \( \beta \) is a formal parameter~\cite[Definition 2.1]{IKEDA201322}.  
Setting \( \beta = 0 \), these polynomials reduce to the Schur \( P \)- and \( Q \)-polynomials, respectively.  
The functions \( GP_\lambda \) and \( GQ_\lambda \) are defined as the ``stable limits'' \( n \to \infty \) of \( GP_\lambda(x_1, \dots, x_n) \) and \( GQ_\lambda(x_1, \dots, x_n) \).  
However, since \( GP_\lambda(x_1, \dots, x_n) \) and \( GQ_\lambda(x_1, \dots, x_n) \) are not stable in the strict sense, the functions \( GP_\lambda \) and \( GQ_\lambda \) do not lie in the usual ring of symmetric functions \( \Lambda \), but rather in the completed ring of symmetric functions \( \widehat{\Lambda} \).

There exists a subring \( G\Gamma\subset \widehat{\Lambda} \), characterized by the \textit{$K$-$Q$-cancellation property}~\cite[Definition 1.1]{IKEDA201322}, that contains all $GP_\lambda$ and $GQ_\lambda$.
Any element of $G\Gamma$ can be expressed as an infinite $\ZZ[\beta]$-linear combination of $GP_\lambda$.
On the other hand, all infinite $\ZZ[\beta]$-linear combinations of $GQ_\lambda$ form a proper subalgebra $G\Gamma_{+}\subsetneq G\Gamma$.
It was shown in \cite{IKEDA201322} that,
at $\beta=-1$ limit,
$G\Gamma$ ({\textit{resp.}~$G\Gamma_{+}$}) is isomorphic to the $K$-theory of the Grassmannian of type $B$ and $D$ ({\textit{resp.}~type $C$}).
Through the isomorphisms, $GP_\lambda$ and $GQ_\lambda$ represent the Schubert class $[\mathcal{O}_{\Omega_\lambda}]$.

It has been reported that the $K$-$P$- and $K$-$Q$-functions possess algebraic properties that extend the characteristics of the Schur $P$- and $Q$-functions.
As with the ordinary Schur $P$- and $Q$-functions, it is expected that they admit a vacuum expectation value representation using vertex operators.
In~\cite{iwao2021neutralfermionic}, the author of the present paper introduced an algebraic characterization of $GQ_\lambda$ by using the \textit{$\beta$-deformed neutral fermion operators}.
This naturally raises the question of whether the other $K$-theoretic functions admit similar expressions.

In this paper, we extend the idea of~\cite{iwao2021neutralfermionic} to obtain new fermionic presentations of $GP_\lambda$ and the dual functions $gq_\lambda$ and $gp_\lambda$.
The central idea of the construction is the introduction of four families of $\beta$-deformed neutral fermion operators $\phi_n^{(\beta)}$, $\phi^{[\beta]}_n$, $\Phi_n^{(\beta)}$, and $\Phi^{[\beta]}_n$ (Definitions \ref{defi:deformed_fermions} and \ref{def:Phi}).
Using these operators, we define several ``$\beta$-deformed'' operators, including the \textit{$\beta$-deformed vertex operators} $e^{\mathcal{H}^{(\beta)}(x)}$ and $e^{\mathcal{H}^{[\beta]}(x)}$.
Figure \ref{fig:beta-objects} summarizes the $\beta$-deformed operators introduced in this paper and describes their behavior in the limit $\beta\to 0$.

\begin{figure}[htbp]
\centering
\begin{tabular}{c|c|c|c|c}
&
$K$-theory
& Dual
 & $\beta\to 0$ limit & 
Notable relations \\\hline\hline
%
%
Operators for $GQ,gq$
&
$\phi^{(\beta)}_n$ & $\phi^{[\beta]}_n$& $\phi_n$ & 
$\phi^{[\beta]}_n=(-1)^n(\phi^{(-\beta)}_{-n})^\ast$
\\\hline
%
%
Operators for $GP,gp$
&
$\Phi^{(\beta)}_n$ & $\Phi^{[\beta]}_n$ & $\frac{1}{2}\phi_n$
& 
$
\begin{aligned}
&[(\Phi_m^{(\beta)})^\ast,\phi_n^{[\beta]}]_+\\[-.1em]
&=[(\Phi_m^{[\beta]})^\ast,\phi_n^{(\beta)}]_+=\delta_{m,n}
\end{aligned}
$\\\hline
%
%
Current operator
&
$b_n^{(\beta)}$ & $b^{[\beta]}_n$ & $b_n$ & 
$b_n^{[\beta]}=(b^{(\beta)}_{-n})^\ast$
\\\hline
%
%
Hamiltonian
&
$\mathcal{H}^{(\beta)}(x)$ & $\mathcal{H}^{[\beta]}(x)$ &
$\mathcal{H}(x)$ & 
\\\hline
%
%
Shift operator
&
$e^{\Theta}$ & $e^\theta$ & $1$ & $\theta=\Theta^\ast$
\end{tabular}

\caption{A list of $\beta$-deformed operators.
Operators with the superscript ${}^{(\beta)}$ are used to construct the $K$-theoretic functions $GP_\lambda$ and $GQ_\lambda$,  
while those with the superscript ${}^{[\beta]}$ are used to construct their duals $gq_\lambda$ and $gp_\lambda$.}
\label{fig:beta-objects}
\end{figure}

Moreover, using $e^{\mathcal{H}^{(\beta)}(x)}$ and $e^{\mathcal{H}^{[\beta]}(x)}$, we introduce two versions of \textit{$\beta$-deformed boson-fermion correspondences}, denoted by $\Omega$ and $\chi$ (\S \ref{sec:boson_fermion}).
These are linear maps that send fermion operators to symmetric functions.
The key observation is that these maps establish a direct connection between vacuum expectation values and the \textit{$K$-theoretic Cauchy kernel} of Nakagawa-Naruse~\cite[Definition 5.3]{nakagawa2016generalized}:
\[
\prod_{i,j}
\frac{1-\overline{x_i}y_j}{1-x_iy_j},\qquad\mbox{where}\quad
\overline{x_i}=\frac{-x_i}{1+\beta x_i}.
\]
Here, the dual functions $gq_\lambda$ and $gp_\lambda$ are defined as the unique symmetric functions that satisfy the following Cauchy identity (see \cite[Definition 1.2]{lewis2024combinatorial}):
\begin{equation}\label{eq:Cauchy_identity}
\sum_{\lambda:\mathrm{strict}}GQ_\lambda(x)gp_\lambda(y)
=
\sum_{\lambda:\mathrm{strict}}GP_\lambda(x)gq_\lambda(y)
=
\prod_{i,j}
\frac{1-\overline{x_i}y_j}{1-x_iy_j}.
\end{equation}

The main result (Theorems \ref{thm:GP}, \ref{thm:gq}, and \ref{thm:fermion_of_gp}) of this paper is summarized as follows:
\begin{thm}
Let $\Omega$ and $\chi$ be the $\beta$-deformed boson-fermion correspondences defined in \S \ref{sec:boson_fermion}.
For a strict partition $\lambda$, there exist four vectors
\[
\ket{\lambda}_{Q},\quad \ket{\lambda}_{P},\quad 
\ket{\lambda}_{q},\ \text{and}\quad \ket{\lambda}_p
\]
(Equations \eqref{eq:vector_Q}, \eqref{eq:vector_P}, \eqref{eq:ket_q}, and \eqref{eq:ket_p})
in the fermion Fock space satisfying
\[
GQ_\lambda=\Omega(\ket{\lambda}_Q),\quad
GP_\lambda=\Omega(\ket{\lambda}_P),\quad
gq_\lambda=\chi(\ket{\lambda}_q),\ \text{and}\quad
gp_\lambda=\chi(\ket{\lambda}_p).
\]
The vectors $\ket{\lambda}_{Q}$, $\ket{\lambda}_{P}$, $\ket{\lambda}_{q}$, and $\ket{\lambda}_{p}$ are constructed using $\phi_n^{(\beta)}$, $\Phi_n^{(\beta)}$, $\phi_n^{[\beta]}$, and $\Phi_n^{[\beta]}$, respectively.
\end{thm}

We briefly outline the construction of the paper.
In Section \ref{sec:Pre}, we give a brief review on the basic concept of neutral fermions.
In Section \ref{sec:beta_def_op}, we define the $\beta$-deformed neutral fermions $\phi^{(\beta)}_n$ and $\phi^{[\beta]}_n$. 
They are defined as an infinite $\QQ(\beta)$-linear combination of the (usual) neutral fermion operators $\phi_n$.
We also define two operators $\Phi^{(\beta)}_n$, $\Phi^{[\beta]}_n$ (\S \ref{sec:beta_ferm}), that satisfy the anti-commutation relations $[(\Phi_m^{(\beta)})^\ast,\phi^{[\beta]}_n ]_+=[(\Phi_m^{[\beta]})^\ast,\phi^{(\beta)}_n ]_+=\delta_{m,n}$ (Lemma \ref{lemma:duality}), where $x\mapsto x^\ast$ denotes the anti-algebra automorphism introduced in \S \ref{sec:Fock_space}.
These relations serve as the $\beta$-deformation of the duality relation $[(\tfrac{1}{2}\phi_m)^\ast,\phi_n]_+=\delta_{m,n}$.

There exist two families of $\beta$-deformed generalization of the power sums $p_n(x)=x_1^n+x_2^n+\cdots$, denoted by $p^{(\beta)}_n$ and $p^{[\beta]}_n$ (\S \ref{sec:beta_def_powersum}).
Importantly, they admit the $K$-theoretic Cauchy identity (Lemma \ref{lemma:Cauchy}):
\begin{equation}\label{eq:Cauchy_2}
\sum_{\lambda:\mathrm{odd}}2^{\ell(\lambda)}z_\lambda^{-1}
p_\lambda^{(\beta)}(x)p_\mu^{[\beta]}(y)=
\prod_{i,j}\frac{1-\overline{x_i}y_j}{1-x_iy_j},\qquad z_\lambda=\prod_{i\geq 1}i^{m_i}\cdot m_i!,
\end{equation}
where $m_i(\lambda)=\sharp\{k\,|\,\lambda_k=i\}$.
From \eqref{eq:Cauchy_identity} and \eqref{eq:Cauchy_2}, we show the existence of an inner product $\langle\cdot,\cdot\rangle$ (Eq.~\eqref{eq:def_of_bilinear_form}) satisfying $\langle p^{(\beta)}_\lambda,p^{[\beta]}_\lambda\rangle=2^{\ell(\lambda)}z_\lambda\delta_{\lambda,\mu}$ for odd paritions $\lambda,\mu$, and $\langle GQ_\lambda,gp_\mu\rangle=\langle GP_\lambda,gq_\mu\rangle=\delta_{\lambda,\mu}$ for strict partitions $\lambda,\mu$.
As a result, the $\beta$-deformed boson-fermion correspondences $\Omega$ and $\chi$, defined in \S \ref{sec:boson_fermion}, preserve the two bilinear forms: the vacuum expectation value of fermionic operators and the inner product of symmetric functions (see \eqref{eq:inner_pres}).

Section \ref{sec:comm_rel} presents some technical lemmas that gives commutation relations of the $\beta$-deformed operators.
In Section \ref{sec:main_part}, we give fermionic presentations of the $K$-theoretic functions $GP_\lambda$ and $GQ_\lambda$ (Theorem \ref{thm:prev_main_theorem}).
The proof is given by comparing the vacuum expectation values of $\Omega(\ket{\lambda}_P)$ and $\Omega(\ket{\lambda}_Q)$ with the generating functions of $GP_\lambda$, $GQ_\lambda$, which were presented in the previous work~\cite{nakagawa2018universalfactrial} by Nakagawa-Naruse.
In Section \ref{sec:gp_and_gq}, we introduce two vectors $\ket{\lambda}_q$ and $\ket{\lambda}_p$ in the fermion Fock space, and present the algebraic descriptions $gq_{\lambda}=\chi(\ket{\lambda}_q)$ (Theorem \ref{thm:gq}) and $gp_{\lambda}=\chi(\ket{\lambda}_p)$ (Theorem \ref{thm:fermion_of_gp}) of the dual $K$-theoretic functions.
These expressions are derived by using the orthonormality ${}_Q\inner{\mu}{\lambda}_p={}_P\inner{\mu}{\lambda}_q=\delta_{\lambda,\mu}$.
As an application, we present generating functions of $gq_\lambda$ (Eq.~\eqref{eq:gq_gen}) and $gp_\lambda$ (Proposition \ref{prop:gen_of_gp}).

\subsection{Related works}

A Pfaffian formula for $GP_\lambda$ and $GQ_\lambda$ was first given in \cite{HUDSON2017115} in the context of the connective $K$-theory of Grassmann bundles.
Nakagawa-Naruse~\cite{nakagawa2016generalized,nakagawa2018universalfactrial,nakagawa2023universal} introduced universal-cohomological generalizations of these functions.
The generating function of $gq_\lambda$ \eqref{eq:gq_gen} was first conjectured in \cite{nakagawa2023universal}.
They also conjectured a combinatorial description of $gq_\lambda$ and $qp_\lambda$ in terms of shifted plane partitions in~\cite{nakagawa2018universalfactrial}, which was later proved by Lewis-Marberg~\cite{lewis2024combinatorial}.
In~\cite{chiu2023expanding}, Chiu-Marberg showed that $\bigoplus_{\lambda}\ZZ[\beta]\cdot GP_\lambda\subsetneq G\Gamma$, the subspace of \textit{finite} linear combinations of $GP_\lambda$, is closed under multiplication.
The same property for $GQ_\lambda$ was proved by Lewis-Marberg~\cite{lewis2024combinatorial}.

Our results also generalize the fermionic description of type $A$ $K$-theoretic functions, including the stable Grothendieck polynomials~\cite{iwao2020freefermion,iwao2022free}, the multi-Schur functions~\cite{iwao2023free}, and the canonical Grothendieck functions~\cite{iwao2024free}.

\subsection*{Acknowledgements}

The author would like to thank Takeshi Ikeda, Hiroshi Naruse, Yang Yi, and Koushik Brahma for their valuable comments on the earlier version.
This work is partially supported by Grant-in-Aid for Scientific Research (C) 19K03065, 22K03239, and 23K03056.

\section{Preliminaries}\label{sec:Pre}

This section gives a brief summary of the neutral fermion Fock space.
We recommend Baker's paper~\cite{baker1995symmetric} and Jimbo-Miwa's paper \cite{jimbo1983solitons} for readers who are interested in this theme.
Let $[A,B]=AB-BA$ be the commutator and $[A,B]_+=AB+BA$ be the anti-commutator.

\subsection{Fock space}\label{sec:Fock_space}

Let $\mathcal{A}$ be the $\QQ(\beta)$-algebra of \textit{neutral fermions} generated by $\{\phi_n\}_{n\in \ZZ}$ satisfying the anti-commutation relation
$
[\phi_m,\phi_n]_+=2(-1)^m\delta_{m+n,0}
$.
In particular, we have $\phi_0^2=1$ and $\phi_n^2=0$ for $n\neq 0$.

Let $\ket{0}$ and $\bra{0}$ denote the \textit{vacuum vectors}: 
\[
\phi_{-n}\ket{0}=0,\quad 
\bra{0}\phi_n=0,\qquad (n>0).
\]
The \textit{Fock space} $\mathcal{F}:=\mathcal{A}\cdot \ket{0}$ is the left $\mathcal{A}$-module generated by $\ket{0}$, and the \textit{dual Fock space}
$\mathcal{F}^\ast:=\bra{0}\cdot \mathcal{A}$ is the right $\mathcal{A}$-module generated by $\bra{0}$.
The \textit{vacuum expectation value} is the unique bilinear form
\[
\mathcal{F}^\ast \otimes_{\QQ(\beta)} \mathcal{F} \to \QQ(\beta);\qquad
\bra{u}\otimes \ket{v}\mapsto \inner{u}{v}
\]
that satisfies (i) $\inner{0}{0} =1$, (ii) $\bra{0}\phi_0\ket{0}=0$, and (iii)
$(\bra{u}\phi_n)\ket{v}=\bra{u}(\phi_n\ket{v})$.
We use the abbreviations 
$\bra{u}X\ket{v}:=(\bra{u}X)\ket{v}=\bra{u}(X\ket{v})$ and
$\langle X\rangle:=\bra{0}X\ket{0}$ for any $X\in \mathcal{A}$.

The Fock space $\mathcal{F}$ is split into two subspaces as $\mathcal{F}=\mathcal{F}_{odd}\oplus \mathcal{F}_{even}$, where $\mathcal{F}_{odd}$ (\textit{resp.}~$\mathcal{F}_{even}$) is the subspace generated by all vectors obtained from $\ket{0}$ by applying odd (\textit{resp.}~even) numbers of $\phi_n$ ($n\geq 0$).
The dual Fock space is also split as $\mathcal{F}^\ast=\mathcal{F}^\ast_{odd}\oplus \mathcal{F}^\ast_{even}$.
If
$\bra{u}\otimes \ket{v}$ is in $ \mathcal{F}^\ast_{even}\otimes \mathcal{F}_{odd}$ or in $\mathcal{F}^\ast_{odd}\otimes \mathcal{F}_{even}$, the vacuum expectation value $\inner{v}{u}$ annihilates automatically.
By restriction, it induces the nondegenerate bilinear form
\begin{equation}\label{eq:even_bilin}
\mathcal{F}_{even}^\ast\otimes_{\QQ(\beta)} \mathcal{F}_{even}\to \QQ(\beta).
\end{equation}

There exists an anti-algebra automorphism $\ast:\mathcal{A}\leftrightarrow\mathcal{A}$; $x\leftrightarrow x^\ast$ defined on the generators by $\phi_n^\ast=(-1)^n\phi_{-n}$ and $(xy)^\ast=y^\ast x^\ast$.
This anti-algebra automorphism induces an involution on the Fock space $\ast:\mathcal{F}\leftrightarrow \mathcal{F}^\ast;v\leftrightarrow v^\ast$ satisfying $\ket{0}^\ast=\bra{0}$.

\subsection{Wick's theorem}

For a $2r\times 2r$ matrix $X$, $\mathrm{Pf}(X)$ denotes the \textit{Pfaffian}
\begin{equation}\label{eq:Pfaffian_def}
\mathrm{Pf}(X)=
\hspace{-1em}
\sum_{\substack{
\sigma(1)<\sigma(3)<\cdots <\sigma(2r-1)\\
\sigma(1)<\sigma(2),\ 
\sigma(3)<\sigma(4),\dots,
\sigma(2r-1)<\sigma(2r)
}}
\hspace{-1em}
\mathrm{sgn}(\sigma)
a_{\sigma(1),\sigma(2)}
a_{\sigma(3),\sigma(4)}\dots a_{\sigma(2r-1),\sigma(2r)}.
\end{equation}
For $n_1,\dots,n_{2r}\in \ZZ$, we have \textit{Wick's theorem} 
\begin{equation}\label{eq:Wick}
\langle
\phi_{n_1}\phi_{n_2}\dots \phi_{n_{2r}}
\rangle
=\mathrm{Pf}
\left(
\langle
\phi_{n_i}\phi_{n_j}
\rangle
\right)_{1\leq i<j\leq 2r}.
\end{equation}

For a strict partition $\lambda=(\lambda_1>\lambda_2>\dots>\lambda_r>0)$ of length $r$, we define the vector $\ket{\lambda}\in \mathcal{F}_{even}$ as
\[
\ket{\lambda}=\begin{cases}
\phi_{\lambda_1}\phi_{\lambda_2}\dots \phi_{\lambda_r}\ket{0} & (r:\mbox{even}),\\
\phi_{\lambda_1}\phi_{\lambda_2}\dots \phi_{\lambda_r}\phi_0\ket{0}  & (r:\mbox{odd}).
\end{cases}
\]
The orthogonality $\langle \mu|\lambda\rangle=2^r\delta_{\lambda,\mu}$ is derived directly from Wick's theorem \eqref{eq:Wick}.

\section{$\beta$-deformed operators and functions}\label{sec:beta_def_op}

In this section, we introduce four families of $\beta$-deformed neutral fermion operators, denoted by $\phi_n^{(\beta)}$, $\phi_n^{[\beta]}$, $\Phi_n^{(\beta)}$, and $\Phi_n^{[\beta]}$.  
The operators $\phi_n^{(\beta)}$ and $\Phi_n^{(\beta)}$ are used to construct the $K$-theoretic functions $GQ_\lambda$ and $GP_\lambda$, while $\phi_n^{[\beta]}$ and $\Phi_n^{[\beta]}$ are used to construct their duals $gq_\lambda$ and $qp_\lambda$.

\subsection{Fermion fields}\label{sec:beta_ferm}

The \textit{neutral fermion field} $\phi(z)$ is the formal series $\phi(z) = \sum_{n \in \ZZ} \phi_n z^n$.  
It defines a $\QQ(\beta)$-linear map $\mathcal{F} \to \mathcal{F}((z))$ by left multiplication, and a $\QQ(\beta)$-linear map $\mathcal{F}^\ast \to \mathcal{F}^\ast((z^{-1}))$ by right multiplication.

\begin{defi}\label{defi:deformed_fermions}
The \textit{$\beta$-deformed fermion fields} 
\[
\phi^{(\beta)}(z)=\sum_{n\in \ZZ}\phi^{(\beta)}_nz^n,\qquad \phi^{[\beta]}(z)=\sum_{n\in \ZZ}\phi^{[\beta]}_nz^n
\]
are formal series defined by the following equations:
\begin{equation}\label{eq:expansions_formal}
\begin{gathered}
\sum_{n=0}^\infty \phi^{(\beta)}_nz^n=
\sum_{n=0}^\infty \phi_n\left(z+\tfrac{\beta}{2}\right)^n,\quad
\sum_{n=1}^\infty \phi^{(\beta)}_{-n}z^{-n}=
\sum_{n=1}^\infty \phi_{-n}\left(\frac{z^{-1}}{1+\frac{\beta}{2}z^{-1}}\right)^n,\\
\sum_{n=1}^\infty \phi^{[\beta]}_nz^n=
\sum_{n=1}^\infty \phi_n\left(\frac{z}{1+\frac{\beta}{2}z}\right)^n,\quad
\sum_{n=0}^\infty \phi^{[\beta]}_{-n}z^{-n}=
\sum_{n=0}^\infty \phi_{-n}\left(z^{-1}+\tfrac{\beta}{2}\right)^n.
\end{gathered}
\end{equation}
The series $\phi^{(\beta)}(z)$ defines a $\QQ(\beta)$-linear map $\mathcal{F}^\ast \to \mathcal{F}^\ast((z^{-1}))$ by right multiplication,  
while $\phi^{[\beta]}(z)$ defines a $\QQ(\beta)$-linear map $\mathcal{F} \to \mathcal{F}((z))$ by left multiplication.
\end{defi}

\begin{rem}
The equations in \eqref{eq:expansions_formal} can be interpreted via the following informal expressions:
\begin{equation}\label{eq:def_of_b_deformed_fermion}
\phi^{(\beta)}(z)=\phi\left(z+\tfrac{\beta}{2}\right),\qquad
\phi^{[\beta]}(z)=\phi\left(\tfrac{z}{1+\frac{\beta}{2}z}\right).
\end{equation}
\end{rem}

If $n>0$, then $\phi^{(\beta)}_n$ and $\phi^{[\beta]}_n$ are linear combinations of $\phi_1,\phi_2,\dots$, while $\phi^{(\beta)}_{-n}$ and $\phi^{[\beta]}_{-n}$ are linear combinations of $\phi_{-1},\phi_{-2},\dots$
Hence, they satisfy the following annihilation rule:
\begin{equation}\label{eq:ann_rule}
\bra{0}\phi_n^{(\beta)}=\bra{0}\phi_n^{[\beta]}=0,\qquad
\phi_{-n}^{(\beta)}\ket{0}=\phi_{-n}^{[\beta]}\ket{0}=0,\qquad (n>0).
\end{equation}

The following lemmas follow immediately from the definitions of $\beta$-deformed operators.
\begin{lemma}
The actions of $\phi_0^{(\beta)}$ and $\phi_0^{[\beta]}$ on the vacuum vectors are expressed as
\begin{equation}\label{eq:vs_phi_0}
\bra{0}\phi_0^{(\beta)}=\bra{0}\phi_0,\quad
\bra{0}\phi_0\phi_0^{(\beta)}=\bra{0},\quad
\phi_0^{[\beta]}\ket{0}=\phi_0\ket{0},\quad
\phi_0^{[\beta]}\phi_0\ket{0}=\ket{0}.
\end{equation}
\end{lemma}

\begin{lemma}[{\cite[\S 9.1]{iwao2021neutralfermionic}}]
\label{lemma:basic_anti_commutation}
We have the anti-commutation relation:
\[
[(\phi^{(\beta)}_m)^\ast, \phi^{[\beta]}_n]_+
=
\begin{cases}
2 & (m=n),\\
\beta & (m=n-1),\\
0 & (\mbox{otherwise}).
\end{cases}
\]
\end{lemma}
\begin{proof}
This lemma is derived directly from $[(\phi_m)^\ast,\phi_n]_+=2\delta_{m,n}$ and Definition \ref{defi:deformed_fermions}.
For details of the proof, see \cite[\S 9.1]{iwao2021neutralfermionic}.
\end{proof}


We also define more two $\beta$-deformed fermion fields as follows:
\begin{defi}\label{def:Phi}
We define the $\beta$-deformed fermion fields
\[
\Phi^{(\beta)}(z)=\sum_{n\in \ZZ}\Phi_n^{(\beta)}z^n,
\quad
\Phi^{[\beta]}(z)=\sum_{n\in \ZZ}\Phi_n^{[\beta]}z^n
\]
by
\[
\Phi^{(\beta)}(z):=\frac{1}{2+\beta z^{-1}}\phi^{(\beta)}(z),\qquad
\Phi^{[\beta]}(z):=\frac{1}{2+\beta z}\phi^{[\beta]}(z).
\]
The series $\Phi^{(\beta)}(z)$ defines a $\QQ(\beta)$-linear map $\mathcal{F}^\ast \to \mathcal{F}^\ast((z^{-1}))$ by right multiplication,  
while $\Phi^{[\beta]}(z)$ defines a $\QQ(\beta)$-linear map $\mathcal{F} \to \mathcal{F}((z))$ by left multiplication.
%
%
\end{defi}

For $n\in \ZZ$, the operators $\Phi_n^{(\beta)}$ and $\Phi_n^{[\beta]}$ can be expanded as
\begin{equation}\label{eq:Phi(beta)exp}
\Phi_n^{(\beta)}=
\frac{1}{2}\sum_{i=0}^\infty \left(-\frac{\beta}{2}\right)^i\phi^{(\beta)}_{n+i},\qquad
\Phi_n^{[\beta]}=
\frac{1}{2}\sum_{i=0}^\infty \left(-\frac{\beta}{2}\right)^i\phi^{[\beta]}_{n-i}.
\end{equation}
From \eqref{eq:ann_rule} and \eqref{eq:Phi(beta)exp}, we have the following annihilation rules
\begin{equation}\label{eq:ann_Phi}
\bra{0}\Phi_n^{(\beta)}=0,\qquad
\Phi_{-n}^{[\beta]}\ket{0}=0,\qquad(n>0).
\end{equation}
and
\begin{equation}\label{eq:ann_Phi_0}
\bra{0}\Phi_0^{(\beta)}=\frac{1}{2}\cdot \bra{0}\phi_0,\qquad
\Phi_{0}^{[\beta]}\ket{0}=\frac{1}{2}\cdot \phi_0\ket{0}.
\end{equation}

However, the vectors $\Phi_{-n}^{(\beta)}\ket{0}$ and $\bra{0}\Phi_{n}^{[\beta]}$ do \textit{not} vanish for $n>0$.
In fact, using the identity
\[
\phi_0
=
\sum_{i=0}^\infty \left(-\frac{\beta}{2}\right)^i\phi^{(\beta)}_{i}
=
\sum_{i=0}^\infty \left(-\frac{\beta}{2}\right)^i\phi^{[\beta]}_{-i},
%
\]
we can rewrite (\ref{eq:Phi(beta)exp}) as
\begin{equation}\label{eq:Phi_expand}
\Phi_{-n}^{(\beta)}
=
\frac{1}{2}\sum_{i=0}^{n-1}\left(-\frac{\beta}{2}\right)^i\phi^{(\beta)}_{-n+i}+\frac{(-\beta)^n}{2^{n+1}}\phi_0,\qquad 
\Phi_{n}^{[\beta]}
=
\frac{1}{2}\sum_{i=0}^{n-1}\left(-\frac{\beta}{2}\right)^i\phi^{[\beta]}_{n-i}+\frac{(-\beta)^n}{2^{n+1}}\phi_0
\end{equation}
for non-negative $n\geq 0$.
From these expressions, we conclude the following equations for $n\geq 0$:
\begin{equation}\label{eq:lack_of_ann_rule}
\Phi_{-n}^{(\beta)}\ket{0}=
\frac{(-\beta)^n}{2^{n+1}}\cdot \phi_0\ket{0},\qquad
\bra{0}\Phi_n^{[\beta]}=\frac{(-\beta)^n}{2^{n+1}}\cdot \bra{0}\phi_0.
\end{equation}

\begin{lemma}[Duality relation]
\label{lemma:duality}
We have the following commutation relations:
\[
[(\Phi^{(\beta)}_m)^\ast, \phi^{[\beta]}_n]_+
=[(\Phi^{[\beta]}_m)^\ast, \phi^{(\beta)}_n]_+
=\delta_{m,n}.
\]
\end{lemma}
\begin{proof}
This lemma follows directly from Lemma \ref{lemma:basic_anti_commutation} and (\ref{eq:Phi(beta)exp}).
\end{proof}

\subsection{Current and Hamiltonian operators}

For any odd integer $m$, the \textit{current operator} $b_m$ is defined as the formal sum
\[
b_m=\frac{1}{4}\sum_{i\in \ZZ}(-1)^i\phi_{-i-m}\phi_i.
\]
\begin{defi}
For any integer $m\neq 0$, we define the \textit{$\beta$-deformed current operators} $b_m^{(\beta)}$ and $b_m^{[\beta]}$ by
\[
b_m^{(\beta)}=
\left.\frac{(X-\frac{\beta}{2})^m-(-X-\frac{\beta}{2})^m }{2}\right|_{X^k\mapsto b_k}\mbox{\quad and \quad}
b_m^{[\beta]}=(b^{(\beta)}_{-m})^\ast.
\]
Here, the expression $\frac{(X-\frac{\beta}{2})^m-(-X-\frac{\beta}{2})^m }{2}$ is understood as a polynomial in $X$ when $m>0$, and as a power series in $X^{-1}$ when $m<0$.
\end{defi}
\begin{example}
For example, we have 
\[
b_1^{(\beta)}=b_1, \quad
b_2^{(\beta)}=-\beta b_1,\quad 
b_3^{(\beta)}=b_3+\frac{3\beta^2}{4}b_1,\quad b_{-1}^{(\beta)}=b_{-1}+\frac{\beta^2}{4}b_{-3}+\frac{\beta^4}{16}b_{-5}+\cdots.
\]
In particular, we have $\lim\limits_{\beta\to 0}b_n^{(\beta)}=0$ when $n$ is even.
\end{example}

Let $p_n(x)=x_1^n+x_2^n+\cdots$ denote the $n$-th power sum.
The \textit{Hamiltonian} $\mathcal{H}(x)$ is the operator defined by
\[
\mathcal{H}(x)=2\sum_{n=1,3,5,\dots} \frac{p_n(x)}{n}b_n.
\]
\begin{defi}
We define the \textit{$\beta$-deformed Hamiltonian operators} $\mathcal{H}^{(\beta)}(x)$ and  $\mathcal{H}^{[\beta]}(x)$ by 
\[
\mathcal{H}^{(\beta)}(x)=2\sum_{n=1}^\infty \frac{p_n(x)}{n}b^{(\beta)}_n,\qquad
\mathcal{H}^{[\beta]}(x)=2\sum_{n=1}^\infty \frac{p_n(x)}{n}b^{[\beta]}_n.
\]
For brevity, we often write 
$\mathcal{H}^{(\beta)}=\mathcal{H}^{(\beta)}(x)$ and
$\mathcal{H}^{[\beta]}=\mathcal{H}^{[\beta]}(x)$.
\end{defi}

\subsection{Bilinear form}\label{sec:beta_def_powersum}

We introduce two families of \textit{$\beta$-deformed power sums}
$p^{(\beta)}_n(x)$ and $p^{[\beta]}_n(x)$ defined as
\[
p^{(\beta)}_n(x):=
\sum_{i=0}^\infty \tbinom{-n}{i}(\tfrac{\beta}{2})^ip_{n+i}(x),\qquad
p^{[\beta]}_n(x):=\sum_{i=1}^n\tbinom{n}{i}(\tfrac{\beta}{2})^ip_i(x).
\]
These equations can be interpreted via the following informal expressions:
\begin{equation}\label{eq:beta_deformed_informal}
p^{(\beta)}_n(x)=p_n(\tfrac{x}{1+\frac{\beta}{2}x}),\qquad
p^{[\beta]}_n(x)=p_n(x+\tfrac{\beta}{2})-p_n(\tfrac{\beta}{2}).
\end{equation}
For any partition $\lambda=(\lambda_1\geq \lambda_2\geq\dots )$, we put
\[
p^{(\beta)}_{\lambda}:=
p^{(\beta)}_{\lambda_1}p^{(\beta)}_{\lambda_2}\cdots,\qquad p^{[\beta]}_{\lambda}:=
p^{[\beta]}_{\lambda_1}p^{[\beta]}_{\lambda_2}\cdots
.
\]

Let  
\[
\widehat{G\Gamma}=\QQ(\beta)[[p_1^{(\beta)},p_3^{(\beta)},p_5^{(\beta)},\dots]]
\] 
be the $\QQ(\beta)$-algebra consisting of all infinite linear combinations of $p_\lambda^{(\beta)}$ with odd partitions $\lambda$, and let 
\[
g\Gamma=\QQ(\beta)[p_1^{[\beta]},p_3^{[\beta]},p_5^{[\beta]},\dots]
\] 
be the $\QQ(\beta)$-algebra consisting of all finite linear combinations of $p_\lambda^{[\beta]}$ with odd partitions $\lambda$.
Then, there exist two natural isomorphisms:
\[
\iota^{(\beta)}:\widehat{\Gamma}\to \widehat{G\Gamma};\quad
p_n\mapsto p_n^{(\beta)},\qquad
\iota^{[\beta]}:\Gamma\to g\Gamma;\quad
p_n\mapsto p_n^{[\beta]},
\]
where $\Gamma=\QQ(\beta)[p_1,p_3,\dots]$ and 
$\widehat{\Gamma}=\QQ(\beta)[[p_1,p_3,\dots]]$.
From \eqref{eq:beta_deformed_informal}, we see that $\iota^{(\beta)}$ coincides with the substitution map $x_i\mapsto \frac{x_i}{1+\frac{\beta}{2}x_i}$.

Let $Q_\lambda\in \Gamma$ be the \textit{Schur $Q$-function}~\cite[\S III.8]{macdonald1998symmetric} for a strict partition $\lambda$. 
\begin{prop}
We have
\begin{equation}\label{eq:beta_deformed_Q}
\iota^{(\beta)}(Q_\lambda)=
\bra{0}e^{\mathcal{H}^{(\beta)}}\ket{\lambda},\qquad
\iota^{[\beta]} (Q_\lambda)=\bra{0}e^{\mathcal{H}^{[\beta]}}\ket{\lambda}.
\end{equation}
\end{prop}
\begin{proof}
These equations can be derived from $Q_\lambda=\bra{0}e^{\mathcal{H}}\ket{\lambda}$ (see, for example \cite[\S 3]{baker1995symmetric}) together with the following equations:
\begin{align}
\mathcal{H}^{(\beta)}=2\sum_{n=1,3,5,\dots}\frac{p^{(\beta)}_n}{n}b_n,\qquad
\mathcal{H}^{[\beta]}=2\sum_{n=1,3,5,\dots}\frac{p^{[\beta]}_n}{n}b_n,
\end{align}
which are shown in \cite[Lemma 7]{iwao2021neutralfermionic}.
\end{proof}

There uniquely exists a bilinear form
\begin{equation}\label{eq:def_of_bilinear_form}
\widehat{G\Gamma}\otimes_{\QQ(\beta)}g\Gamma\to \QQ(\beta);\qquad
f\otimes g\mapsto 
\left\langle
f,g
\right\rangle
\end{equation} 
satisfying $
\langle
p_\lambda^{(\beta)}, p_\mu^{[\beta]}
\rangle=
2^{-\ell(\lambda)} z_\lambda\delta_{\lambda,\mu}$ for all odd partitions $\lambda,\mu$.
Through the bilinear form \eqref{eq:def_of_bilinear_form}, $\widehat{G\Gamma}$ can be identified with the linear space $\mathrm{Hom}_{\QQ(\beta)}(g\Gamma,\QQ(\beta))$.

\begin{lemma}[{\cite[\S 8]{iwao2021neutralfermionic}}]\label{lemma:Cauchy}
We have the following $K$-theoretic Cauchy identity
\[
\sum_{\lambda:\mathrm{odd}}2^{\ell(\lambda)}z_\lambda^{-1}
p_\lambda^{(\beta)}(x)p_\mu^{[\beta]}(y)=
\prod_{i,j}\frac{1-\overline{x_i}y_j}{1-x_iy_j},\qquad
\mbox{where}\quad \overline{x}=-\frac{x}{1+\beta x}.
\]
\end{lemma}
\begin{proof}
This lemma follows from the (ordinary) Cauchy kernel
\[
\sum_{\lambda:\mathrm{odd}}2^{\ell(\lambda)}z_\lambda^{-1}
p_\lambda(x)p_\mu(y)=
\prod_{i,j}\frac{1+x_iy_j}{1-x_iy_j}
\]
by applying the substitutions $x_i\mapsto \frac{x_i}{1+\frac{\beta}{2}x_i}$ and $p_n(y)\mapsto p_n(y+\frac{\beta}{2})-p_n(\frac{\beta}{2})$.
For details, see \cite[\S 8]{iwao2021neutralfermionic}.
\end{proof}

Let
$Q^{(\beta)}_\lambda:=\iota^{(\beta)}(Q_\lambda)$ and 
$Q^{[\beta]}_\lambda:=\iota^{[\beta]}(Q_\lambda)$.

\begin{lemma}\label{lemma:beta_Q_dual}
For any strict partitions $\lambda,\mu$, we have
$
\langle
Q^{(\beta)}_\lambda,
Q^{[\beta]}_\mu
\rangle
=2^{\ell(\lambda)}\delta_{\lambda,\mu}
$.
\end{lemma}
\begin{proof}
Let $\langle f,g\rangle':=\langle \iota^{(\beta)}(f),\iota^{[\beta]}(g)\rangle$ for $f,g\in \Gamma$.
Then, the bilinear form $\langle \cdot,\cdot \rangle'$ coincides with the Hall inner product on $\Gamma$~\cite[\S III.8]{macdonald1998symmetric}.
The lemma follows from the orthogonality $\langle Q_\lambda,Q_\mu\rangle'=2^{\ell(\lambda)}\delta_{\lambda,\mu}$.
\end{proof}

\subsection{Boson-fermion correspondence}
\label{sec:boson_fermion}

Let $\Omega_0$ and $\chi$ be the $\QQ(\beta)$-linear maps defined by
\[
\begin{gathered}
\Omega_0:\mathcal{F}_{even}\to \widehat{G\Gamma};\qquad \ket{v}\mapsto \bra{0}e^{\mathcal{H}^{(\beta)}}\ket{v},\\
\chi:\mathcal{F}_{even}\to g\Gamma;\qquad \ket{v}\mapsto \bra{0}e^{\mathcal{H}^{[\beta]}}\ket{v}.
\end{gathered}
\]
By \eqref{eq:beta_deformed_Q} and Lemma \ref{lemma:beta_Q_dual}, we see that these linear maps satisfy the relation
\begin{equation}\label{eq:fundamental_rel_of_Omegas}
\inner{u^\ast}{v}=
\big\langle \Omega_0(\ket{u}), \chi(\ket{v}) \big\rangle
, \qquad\text{where}\ \bra{u^\ast}=(\ket{u})^\ast,
\end{equation}
for all elements $\ket{u},\ket{v}\in \mathcal{F}_{even}$.
Note that $\chi$ is bijective, while $\Omega_0$ is not, because $\widehat{G\Gamma}$ is too large a vector space.

To modify the map $\Omega_0$ into a bijection, we consider an extension of the vector space $\mathcal{F}_{even}$.
Let $\widehat{\mathcal{F}}_{even}:=
\mathrm{Hom}_{\QQ(\beta)}(\mathcal{F}^\ast_{even},\QQ(\beta))
$ be the dual space of $\mathcal{F}^\ast_{even}$.
For any $\varphi\in \widehat{\mathcal{F}}_{even}$,
there exists a unique element $x_\varphi\in \widehat{G\Gamma}$ satisfying the relation
\begin{equation}\label{eq:def_of_Omega_rel}
\big\langle
x_\varphi,\chi(\ket{v})
\big\rangle=
\varphi(\bra{v^\ast})
\end{equation}
for all $\ket{v}\in\mathcal{F}_{even}$, since the bilinear form \eqref{eq:def_of_bilinear_form} is non-degenerate.
This defines a unique linear map
\[
\Omega:\widehat{\mathcal{F}}_{even}\to \widehat{G\Gamma};\quad \varphi\mapsto x_\varphi.
\] 

On the other hand, through the bilinear form \eqref{eq:even_bilin}, $\mathcal{F}_{even}$ can be identified with a subspace of $\widehat{\mathcal{F}}_{even}$.
Comparing \eqref{eq:fundamental_rel_of_Omegas} with \eqref{eq:def_of_Omega_rel}, we find that the restriction of $\Omega$ to $\mathcal{F}_{even}$ coincides with $\Omega_0$.
Using \eqref{eq:fundamental_rel_of_Omegas} again, we obtain the identity
\begin{equation}\label{eq:inner_pres}
\inner{u^\ast}{v}=
\big\langle \Omega(\ket{u}), \chi(\ket{v}) \big\rangle
\end{equation}
for all $\ket{u}\in \widehat{\mathcal{F}}_{even}$ and $\ket{v}\in \mathcal{F}_{even}$.
One can verify that $\Omega$ is bijective.

We refer to the bijections $\Omega$ and $\chi$ as the \textit{$\beta$-deformed boson-fermion correspondences}.

\section{Algebraic relations of $\beta$-deformed operators}\label{sec:comm_rel}

In this section, we present several lemmas describing commutation relations among the $\beta$-deformed fermionic and current operators.
Most of the proofs can be found in the previous paper~\cite{iwao2021neutralfermionic}.

The \textit{$K$-theoretic addition} and \textit{subtraction} are binary operators $\oplus$ and $\ominus$ defined by
\[
x\oplus y=x+y+\beta xy,\quad
x\ominus y=\frac{x-y}{1+\beta y}.
\]
In particular, we have $\overline{t}=\frac{-t}{1+\beta t}=0\ominus t$.

Let $\Theta$ and $\theta$ be the operators
\[
\Theta=2
\sum_{n=1,3,5,\dots}\left(\frac{\beta}{2}\right)^n\frac{b_{-n}}{n},\qquad
\theta=
\Theta^\ast=
2
\sum_{n=1,3,5,\dots}\left(\frac{\beta}{2}\right)^n\frac{b_{n}}{n}.
\]

\begin{lemma}\label{lemma:comm_rels_0}
We have the following commutation relations:
\begin{enumerate}
\item\label{item:1-1} $\bra{0}e^{\Theta}=\bra{0}$.
\item\label{item:1-2} $e^{-\Theta}\phi^{(\beta)}(z)e^{\Theta}
=\frac{1}{1+\beta z^{-1}}\cdot \phi^{(\beta)}(z)
$.
\item\label{item:1-3} $e^{-\Theta}e^{\mathcal{H}^{(\beta)}}e^{\Theta}=\prod_{i}(1+\beta x_i)\cdot e^{\mathcal{H}^{(\beta)}}$.
\item\label{item:1-4}
$e^{\mathcal{H}^{(\beta)}}
\phi^{(\beta)}(z)e^{-\mathcal{H}^{(\beta)}}=\prod_{i}\frac{z^{-1}\oplus x_i}{z^{-1}-x_i}\cdot \phi^{(\beta)}(z)$.
\item\label{item:1-5}
$
\left\langle
\phi^{(\beta)}(z)\phi^{(\beta)}(w)
\right\rangle
=\frac{w^{-1}-z^{-1}}{w^{-1}\oplus z^{-1}}
$,
where $\frac{w^{-1}-z^{-1}}{w^{-1}\oplus z^{-1}}$ is understood as 
\[
\begin{aligned}
&%
\textstyle\frac{1-wz^{-1}}{1+wz^{-1}+\beta z^{-1}}=
1-(2w+\beta)z^{-1}+(2w^{2}+3\beta w+\beta^2)z^{-2}-\cdots.
\end{aligned}
\]
This is an element of the field $\QQ(\beta)((w^{-1}))((z^{-1}))$ \footnote{The field $\QQ(\beta)((w^{-1}))((z^{-1}))=\left\{\QQ(\beta)((w^{-1}))\right\}((z^{-1}))$ is not the same as $\QQ(\beta)((z^{-1}))((w^{-1}))$.
In fact, the former contains $1+wz^{-1}+w^2z^{-2}+w^{3}z^{-3}+\cdots$, while the latter does not.
}.
\end{enumerate}
\end{lemma}
\begin{proof}
\eqref{item:1-1} follows from $\bra{0}b_{-n}=0$ for $n>0$.
\eqref{item:1-2} and \eqref{item:1-3} are given in \cite[\S 4.3]{iwao2021neutralfermionic}.
\eqref{item:1-4} follows from the following equation, which is given in \cite[\S 4.2]{iwao2021neutralfermionic},
\[
e^{\mathcal{H}^{(\beta)}}\phi^{(\beta)}(z)e^{\mathcal{H}^{-(\beta)}}
=
\exp\left(\sum_{n=1}^\infty\frac{p_n(x)}{n}(z^n-(-z-\beta)^n)\right)\phi^{(\beta)}(z)
\]
and the identity $\exp\left(\sum_{n=1}^\infty\frac{p_n(x)}{n}A^n\right)=\prod_i\frac{1}{1-x_iA}$.
\eqref{item:1-5} is given in \cite[\S 3.2]{iwao2021neutralfermionic}.
\end{proof}

\begin{lemma}\label{lemma:commutation_rels}
We have the following commutation relations:
\begin{enumerate}
\item\label{item:2-1} $e^{\theta}\ket{0}=\ket{0}$.
\item \label{item:2-2}
$e^{-\theta}\phi^{[\beta]}(z)e^{\theta}=\frac{1}{1+\beta z}\cdot \phi^{[\beta]}(z)$.
\item\label{item:2-3} $e^\theta$ and $e^{\mathcal{H}^{[\beta]}}$ commute with each other.
\item \label{item:2-4}
$e^{\mathcal{H}^{[\beta]}}\phi^{[\beta]}(z)e^{-\mathcal{H}^{[\beta]}}
=
\prod_{i}\frac{1-x_i\overline{z}}{1-x_iz}\cdot
\phi^{[\beta]}(z)
$.
\item \label{item:2-5}
$
\left\langle
\phi^{[\beta]}(z)\phi^{[\beta]}(w)
\right\rangle
=\frac{z-w}{z\oplus w},
$
where $\frac{z-w}{z\oplus w}$ is understood as
\[
\begin{aligned}
&\textstyle
\frac{1-wz^{-1}}{1+wz^{-1}+\beta w}=
1-(2z^{-1}+\beta)w+(2z^{-2}+3\beta z^{-1}+\beta^2)w^2-\cdots.
\end{aligned}
\]
This is an element of the field $\QQ(\beta)((z))((w))$.
\end{enumerate}
\end{lemma}
\begin{proof}
\eqref{item:2-1} follows from $b_{n}\ket{0}=0$ for $n>0$.
\eqref{item:2-2} and \eqref{item:2-5} are given in \cite[\S 10.1]{iwao2021neutralfermionic}.
\eqref{item:2-3} follows from $[b_m,b_n]=0$ for $m,n>0$.
\eqref{item:2-4} is given in \cite[\S 9.1]{iwao2021neutralfermionic}.
\end{proof}

\begin{cor}\label{cor:sublemma}
We have
\begin{enumerate}
\item\label{item:aa} $e^{\theta}\phi_n^{[\beta]}e^{-\theta}=\phi^{[\beta]}_n+\beta \phi^{[\beta]}_{n-1}$,
\item
$e^{-\theta}\phi_n^{[\beta]}e^{\theta}=\phi^{[\beta]}_n-\beta \phi^{[\beta]}_{n-1}+\beta^2\phi^{[\beta]}_{n-2}-\cdots$,
\item\label{item:bb} 
$e^{\theta}(\phi_n^{(\beta)})^\ast e^{-\theta}=
(\phi_n^{(\beta)}
-\beta \phi_{n+1}^{(\beta)}
+\beta^2 \phi_{n+2}^{(\beta)}-\cdots)^\ast
$.
\end{enumerate}
\end{cor}
\begin{proof}
The equations (1) and (2) follow from Lemma \ref{lemma:commutation_rels} \eqref{item:2-2}.
The equation (3) follows from Lemma \ref{lemma:comm_rels_0} \eqref{item:1-2}.
\end{proof}

\section{$GP_\lambda$ and $GQ_\lambda$-functions}\label{sec:main_part}

Let $\ket{\lambda}_Q$ be the element of $ \widehat{\mathcal{F}}_{even}$ defined by
\begin{equation}\label{eq:vector_Q}
\ket{\lambda}_Q=
\begin{cases}
\phi^{(\beta)}_{\lambda_1}e^{\Theta}\phi^{(\beta)}_{\lambda_2}e^{\Theta}\cdots
\phi^{(\beta)}_{\lambda_r}e^{\Theta}\ket{0} & (r:\mbox{even}),\\
\phi^{(\beta)}_{\lambda_1}e^{\Theta}\phi^{(\beta)}_{\lambda_2}e^{\Theta}\cdots
\phi^{(\beta)}_{\lambda_r}e^{\Theta}
\phi^{(\beta)}_{0}e^{\Theta}
\ket{0}  & (r:\mbox{odd}).
\end{cases}
\end{equation}
Note that $\ket{\lambda}_Q$ is \textit{not} an element of $\mathcal{F}_{even}$ if $\lambda\neq \emptyset$.
The main theorem of the previous paper~\cite{iwao2021neutralfermionic} is described as follows:
\begin{thm}[{\cite[\S 7]{iwao2021neutralfermionic}}]\label{thm:prev_main_theorem}
The $K$-theoretic $Q$-function $GQ_\lambda(x)$ is expressed as 
\begin{equation*}
GQ_\lambda(x)=\bra{0}e^{\mathcal{H}^{(\beta)}}\ket{\lambda}_Q
=\Omega(\ket{\lambda}_Q).
\end{equation*}
\end{thm}
In \cite{iwao2021neutralfermionic}, Theorem \ref{thm:prev_main_theorem} was proved by comparing the vacuum expectation value $\bra{0}e^{\mathcal{H}^{(\beta)}}\ket{\lambda}_Q$ with a Pfaffian formula due to Hudson-Ikeda-Matsumura-Naruse~\cite{HUDSON2017115}.
For later use, however, we now provide an alternative proof using the generating function 
\begin{equation}\label{eq:gen_of_GQ}
GQ_\lambda(x)=[u_1^{-\lambda_1}\dots u_{r}^{-\lambda_{r}}]
\prod_{i=1}^r
\frac{1}{1+\beta u_i}
\prod_{i,j}\frac{u_i\oplus x_j}{u_i\ominus x_j}
\prod_{1\leq i<j\leq r}\frac{u_j\ominus u_i}{u_j\oplus u_i},
\end{equation}
which was introduced by Nakagawa-Naruse \cite[\S 5.2]{nakagawa2018universalfactrial}.
Here, $[u_1^{n_1}\dots u_r^{n_r}]F(u_1,\dots,u_r)$ is the coefficient of the monomial $u_1^{n_1}\dots u_r^{n_r}$ in the expansion of $F$.
The rational function on the right hand side of \eqref{eq:gen_of_GQ} is understood as an element of the field
\[
\QQ(\beta)((u_r))\dots ((u_2))((u_1))[[x_1,x_2,\dots]]
\]
via the Laurent expansion on the domain $\{|x_j|<|u_1|<|u_2|<\dots<|u_r|<|\beta^{-1}|:\forall j\}$.

Note that the equation
\begin{equation}\label{eq:substitution_GQ}
GQ_{(\lambda,0)}(x)
=
GQ_{\lambda}(x)
\end{equation}
is not trivial from \eqref{eq:gen_of_GQ} because the right hand side of \eqref{eq:gen_of_GQ} does not admit the simple substitution $u_{r}=0$.
In fact, we need some tedious computation to derive \eqref{eq:substitution_GQ}, which we will explain in the Appendix.

\begin{proof}[Proof of Theorem \ref{thm:prev_main_theorem}]
Without loss of generality, we may assume that $r$ is even by appending $\lambda_{r+1}=0$ to the end of the strict partition if necessary. 
Let $\mathcal{GQ}(z_1,\dots,z_r)$ be the formal series defined by
\[
\mathcal{GQ}(z_1,\dots,z_r):=\left\langle
e^{\mathcal{H}^{(\beta)}}
\phi^{(\beta)}(z_1)e^{\Theta}
\cdots
\phi^{(\beta)}(z_{r})e^{\Theta}
\right\rangle.
\]
Then, we find that
\begin{equation}\label{eq:GQ_to_GQvector}
\bra{0}e^{\mathcal{H}^{(\beta)}}\ket{\lambda}_Q=[z_1^{\lambda_1}\dots z_r^{\lambda_r}]
\left(\mathcal{GQ}(z_1,\dots,z_r)\right).
\end{equation}
An explicit expression of $\mathcal{GQ}(z_1,\dots,z_r)$ can be calculated by using the anti-commutation relations given in Lemma \ref{lemma:comm_rels_0}.
In fact, we have
\begin{align*}
&
\left\langle
e^{\mathcal{H}^{(\beta)}}
\phi^{(\beta)}(z_1)e^{\Theta}
\cdots
\phi^{(\beta)}(z_{r})e^{\Theta}
\right\rangle\\
&=
\prod_j{(1+\beta x_j)^{r}}
\prod_i\frac{1}{(1+\beta z_i^{-1})^{r-i+1}}
\prod_{i,j}\frac{z_i^{-1}\oplus x_j}{z_i^{-1}-x_j}
\left\langle
\phi^{(\beta)}(z_1)
\cdots
\phi^{(\beta)}(z_{r})
\right\rangle
\qquad (\text{Lemma \ref{lemma:comm_rels_0} \eqref{item:1-1}--\eqref{item:1-4}})
\\
&=
\prod_j{(1+\beta x_j)^{r}}
\prod_i\frac{1}{(1+\beta z_i^{-1})^{r-i+1}}
\prod_{i,j}\frac{z_i^{-1}\oplus x_j}{z_i^{-1}-x_j}
\cdot
\mathrm{Pf}
\left(
\langle
\phi^{(\beta)}(z_i)\phi^{(\beta)}(z_j)
\rangle
\right)_{1\leq i<j\leq r}
\quad (\mathrm{Eq.~\eqref{eq:Wick}})
\allowdisplaybreaks
\\
&=
\prod_j{(1+\beta x_j)^{r}}
\prod_i\frac{1}{(1+\beta z_i^{-1})^{r-i+1}}
\prod_{i,j}\frac{z_i^{-1}\oplus x_j}{z_i^{-1}-x_j}
\cdot
\mathrm{Pf}
\left(
\frac{z_j^{-1}-z_i^{-1}}{z_j^{-1}\oplus z_i^{-1}}
\right)_{1\leq i<j\leq r}
\quad (\mathrm{Lemma~\ref{lemma:comm_rels_0}~\eqref{item:1-5}})
\allowdisplaybreaks
\\
&
\stackrel{(\ast)}{=}
\prod_j{(1+\beta x_j)^{r}}
\prod_i\frac{1}{(1+\beta z_i^{-1})^{r-i+1}}
\prod_{i,j}\frac{z_i^{-1}\oplus x_j}{z_i^{-1}-x_j}
\prod_{1\leq i<j\leq r}
\frac{z_j^{-1}-z_i^{-1}}{z_j^{-1}\oplus z_i^{-1}}\\
&=
\prod_i\frac{1}{1+\beta z_i^{-1}}
\prod_{i,j}\frac{z_i^{-1}\oplus x_j}{z_i^{-1}\ominus x_j}
\prod_{1\leq i<j\leq r}
\frac{z_j^{-1}\ominus z_i^{-1}}{z_j^{-1}\oplus z_i^{-1}}.
\end{align*}
For the equality ($\ast$), we used the following formula of Ikeda-Naruse~\cite[Lemma 2.4]{IKEDA201322}:
\begin{equation}\label{eq:IN-formula}
\mathrm{Pf}\left(
\frac{T_i-T_j}{T_i\oplus T_j}
\right)_{1\leq i<j\leq r}=
\prod_{1\leq i<j\leq r}\frac{T_i-T_j}{T_i\oplus T_j}.
\end{equation}
By substituting $z_i=u_i^{-1}$ and comparing \eqref{eq:gen_of_GQ} with \eqref{eq:GQ_to_GQvector}, we obtain the desired equation $GQ_\lambda(x)=\bra{0}e^{\mathcal{H}^{(\beta)}}\ket{\lambda}_Q$.
\end{proof}

A similar result holds for the $GP$-function $GP_\lambda(x)$.
Let $\ket{\lambda}_P$ be the element of $\widehat{\mathcal{F}}_{even}$ defined by
\begin{equation}\label{eq:vector_P}
\ket{\lambda}_P=
\begin{cases}
\Phi^{(\beta)}_{\lambda_1}e^{\Theta}
\Phi^{(\beta)}_{\lambda_2}e^{\Theta}\cdots
\Phi^{(\beta)}_{\lambda_r}e^{\Theta}\ket{0} & (r:\mbox{even}),\\
\Phi^{(\beta)}_{\lambda_1}e^{\Theta}
\Phi^{(\beta)}_{\lambda_2}e^{\Theta}\cdots
\Phi^{(\beta)}_{\lambda_r}e^{\Theta}
\phi^{(\beta)}_{0}e^{\Theta}
\ket{0}  & (r:\mbox{odd}).
\end{cases}
\end{equation}

\begin{thm}\label{thm:GP}
The $K$-theoretic $P$-function $GP_\lambda(x)$ is expressed as
\[
GP_\lambda(x)=\bra{0}e^{\mathcal{H}^{(\beta)}}\ket{\lambda}_P=\Omega(\ket{\lambda}_P).
\]
\end{thm}

\begin{proof}
This theorem is given by comparing the definition of $\Phi^{(\beta)}(z)$ (Definition \ref{def:Phi})
and the generating function 
\begin{align}\label{eq:quote_GP}
&GP_\lambda(x)=[u_1^{-\lambda_1}\dots u_r^{-\lambda_r}]
\prod_{i}
\frac{1}{2+\beta u_i}
\frac{1}{1+\beta u_i}
\prod_{i,j}\frac{u_i\oplus x_j}{u_i\ominus x_j}
\prod_{1\leq i<j\leq r}\frac{u_j\ominus u_i}{u_j\oplus u_i},
\end{align}
which was given in \cite[\S 4.1]{nakagawa2018universalfactrial}\footnote{
In \cite[\S 4.1]{nakagawa2018universalfactrial}, Nakagawa-Naruse presented the generating function
\[
HP_{\lambda}(x)
=HP_{\lambda}(x|\mathbf{0})
=[\mathbf{u}^{-\lambda}]
\prod_i\frac{u_i}{u_i+_{\LL}[t]\overline{u_i}}\cdot
\frac{1}{\mathcal{I}^{\LL}(u_i)}\cdot
\prod_{i,j}\frac{u_i+_{\LL}[t]\overline{x_j}}{u_i+_{\LL}\overline{x_j}}
\cdot
\prod_{j<i}
\frac{u_i+_{\LL}\overline{u_j}}{u_i+_{\LL}[t]\overline{u_j}}
\]
of the universal $P$-function $HP_\lambda(x)$.
Here, $HP_\lambda(x|b)$ is the universal factorial $P$-function.
For the $K$-theory setting, we substitute
$HP_\lambda(x)\mapsto GP_\lambda(x)$,
$x+_{\LL}y\mapsto x\oplus y$,
$\mathcal{I}^{\LL}(u)\mapsto 1+\beta u$,
$[t]\overline{u}\mapsto u$
to obtain \eqref{eq:quote_GP}.
}.
\end{proof}

\section{$gp_\lambda$ and $gq_\lambda$-functions}\label{sec:gp_and_gq}

In this section, we provide a fermionic presentation of the dual $K$-theoretic functions $gq_\lambda$ and $gp_\lambda$ defined in \eqref{eq:Cauchy_identity}.
By comparing \eqref{eq:Cauchy_identity} with Lemma \ref{lemma:Cauchy}, we see that these functions are the unique elements of $g\Gamma$ satisfying the duality relations
\[
\langle
GQ_\lambda,gp_\mu
\rangle 
=
\langle
GP_\lambda,gq_\mu
\rangle 
=\delta_{\lambda,\mu},
\]
where $\langle \cdot,\cdot\rangle$ denotes the bilinear form defined in \eqref{eq:def_of_bilinear_form}.

\subsection{Fermionic presentation of $gq$}\label{sec:Fermion_of_gq}

For a strict partition $\lambda$, let $\ket{\lambda}_q$ be the element of $\mathcal{F}_{even}$ defined by 
\begin{equation}\label{eq:ket_q}
\ket{\lambda}_q:=
\begin{cases}
\phi^{[\beta]}_{\lambda_1}e^{-\theta}
\phi^{[\beta]}_{\lambda_2}e^{-\theta}
\dots
\phi^{[\beta]}_{\lambda_r}e^{-\theta} 
\ket{0} & (r:\mbox{even}),\\
\phi^{[\beta]}_{\lambda_1}e^{-\theta}
\phi^{[\beta]}_{\lambda_2}e^{-\theta}
\dots
\phi^{[\beta]}_{\lambda_r}e^{-\theta} 
\phi^{[\beta]}_{0}e^{-\theta} 
\ket{0} & (r:\mbox{odd}).
\end{cases}
\end{equation}
\begin{thm}\label{thm:gq}
The dual $K$-theoretic $Q$-function $gq_\lambda(x)$ is expressed as
\[
gq_\lambda(x)=\bra{0}e^{\mathcal{H}^{[\beta]}}\ket{\lambda}_q=
\chi(\ket{\lambda}_q).
\]
\end{thm}

Let ${}_P\bra{\mu}:=(\ket{\mu}_P)^\ast$.
In order to prove Theorem \ref{thm:gq}, it suffices to show the duality relation
\begin{equation}\label{eq:aim_to_show_1}
{}_P\inner{\mu}{\lambda}_q=\delta_{\lambda,\mu}.
\end{equation}
To this end, we introduce two auxiliary vectors $(\lambda|\in \mathcal{F}^\ast_{even}$ and $|\lambda)\in \mathcal{F}_{even}$ defined by
\[
\begin{aligned}
\mbra{\lambda}
&=
\begin{cases}
\bra{0}
e^\theta (\Phi^{(\beta)}_{\lambda_r})^\ast
\dots
e^\theta (\Phi^{(\beta)}_{\lambda_2})^\ast
e^\theta (\Phi^{(\beta)}_{\lambda_1})^\ast & (\lambda_r>0),\\
\bra{0}
e^\theta (\phi^{(\beta)}_{0})^\ast
e^\theta (\Phi^{(\beta)}_{\lambda_{r-1}})^\ast
\dots
e^\theta (\Phi^{(\beta)}_{\lambda_2})^\ast
e^\theta (\Phi^{(\beta)}_{\lambda_1})^\ast & (\lambda_r=0),
\end{cases}\\
\mket{\lambda}
&=
\phi^{[\beta]}_{\lambda_1}e^{-\theta}
\phi^{[\beta]}_{\lambda_2}e^{-\theta}
\dots
\phi^{[\beta]}_{\lambda_r}e^{-\theta}\ket{0}
\end{aligned}
\]
for a strictly decreasing sequence $\lambda=(\lambda_1>\lambda_2>\dots>\lambda_r\geq 0)$.
The desired equation \eqref{eq:aim_to_show_1} is therefore equivalent to
\begin{equation}\label{eq:to_prove_1}
(\mu|\lambda)=\delta_{\lambda,\mu}.
\end{equation}

\begin{lemma}\label{lemma:hojyo_1}
We have the following relations:
\begin{itemize}
\item[(A)]
When $n>0$, we have $\bra{0}e^{\theta}(\phi_0^{(\beta)})^\ast\cdot \phi^{[\beta]}_n=0$.
\item[(B)]
When $\mu\neq \emptyset$ and $n>\mu_1$, we have $\mbra{\mu}\phi^{[\beta]}_n=0$.
\item[(C)]
When $\lambda\neq \emptyset$ and $m>\lambda_1$, we have
$\Phi^{(\beta)}_m\mket{\lambda}=0$.
\end{itemize}
\end{lemma}
\begin{proof}
(A):
When $n=1$, we have 
\begin{align*}
\bra{0}e^{\theta}(\phi_0^{(\beta)})^\ast\cdot \phi^{[\beta]}_1
&=
\bra{0}e^{\theta}\left\{
[(\phi_0^{(\beta)})^\ast,\phi^{[\beta]}_1 ]_+-\phi^{[\beta]}_1(\phi_0^{(\beta)})^\ast
\right\}\\
&
=
\beta\cdot \bra{0}e^{\theta}-\bra{0}e^{\theta}\phi^{[\beta]}_1(\phi^{(\beta)}_0)^\ast
\qquad (\mathrm{Lemma\   \ref{lemma:basic_anti_commutation}})\\
&
=
\beta\cdot \bra{0}e^{\theta}-\bra{0}(\phi^{[\beta]}_1+\beta \phi^{[\beta]}_0)(\phi^{(\beta)}_0-\beta\phi^{(\beta)}_1+\beta^2\phi^{(\beta)}_2-\cdots)^\ast e^\theta
\qquad (\text{Corollary~\ref{cor:sublemma}})
\\
&=
\beta\cdot \bra{0}e^{\theta}-
\beta\cdot \bra{0}e^{\theta}\qquad (\mathrm{Eqs.~\eqref{eq:ann_rule},\eqref{eq:vs_phi_0}} )\\
&=0.
\end{align*}
When $n>1$, (A) is proved immediately from $[(\phi^{(\beta)}_0)^\ast,\phi^{[\beta]}_n]_+=0$ and the annihilation rule \eqref{eq:ann_rule}.

(B):
The claim (B) is shown by induction on $s\geq 1$.
When $s=1$ and $\mu_1=0$, (B) is nothing but (A).
When $s=1$ and $\mu_1>0$, we have $n>1$ and
\[
\begin{aligned}
\mbra{\mu}\phi^{[\beta]}_n
&=
\bra{0}e^\theta (\Phi^{(\beta)}_{\mu_1})^\ast\phi^{[\beta]}_n
=
-\bra{0}e^\theta \phi^{[\beta]}_n(\Phi^{(\beta)}_{\mu_1})^\ast
\end{aligned}
\]
from Lemma \ref{lemma:duality}.
By Corollary \ref{cor:sublemma} and \eqref{eq:ann_rule}, we obtain
\[
\begin{aligned}
-\bra{0}e^\theta \phi^{[\beta]}_n(\Phi^{(\beta)}_{\mu_1})^\ast
=
-\bra{0}(\phi_n^{[\beta]}+\beta \phi^{[\beta]}_{n-1})
e^{\theta}
(\Phi^{(\beta)}_{\mu_1})^\ast=0,
\end{aligned}
\]
which implies (B).
For general $s>1$, let $\mu'=(\mu_2>\dots>\mu_s\geq 0)$.
Then, we have $n-1>\mu_2$ and
\[
\begin{aligned}
\mbra{\mu}\phi^{[\beta]}_n
&=
\mbra{\mu'}e^\theta (\Phi^{(\beta)}_{\mu_1})^\ast\phi^{[\beta]}_n
=
-\mbra{\mu'}e^\theta \phi^{[\beta]}_n(\Phi^{(\beta)}_{\mu_1})^\ast
=
-\mbra{\mu'}(\phi_n^{[\beta]}+\beta \phi^{[\beta]}_{n-1})
e^{\theta}
(\Phi^{(\beta)}_{\mu_1})^\ast=0,
\end{aligned}
\]
where the last equality follows from the induction hypothesis.

(C):
The claim (C) follows from the fact that $|\lambda)$ is expressed as a $\QQ(\beta)$-linear combination of vectors of the form
$\phi^{[\beta]}_{n_1}\phi^{[\beta]}_{n_2}\dots 
\phi^{[\beta]}_{n_r}\ket{0}$ with $m>n_1>n_2>\dots>n_r\geq 0$
(see Corollary \ref{cor:sublemma} (2)).
\end{proof}

\begin{proof}[Proof of Theorem \ref{thm:gq}]
Let $E=( \mu|\lambda)$.
Note that $E$ is automatically $0$ when $r+s$ is odd.

(i) 
When $(s,r)=(0,0)$, we have $E=\langle 0|0\rangle=1$.
When $(s,r)=(0,1)$, we have $E=0$ since $r+s$ is odd.
When $s=0$ and $r\geq 2$, we have $\lambda_1>0$.
Hence, by the annihilation rule \eqref{eq:ann_rule}, we have $E=(\emptyset|\lambda)=0$.

(ii) 
When $r=0$, we can show that $E=0$ in a similar manner to (i).

(iii)
For general $(s,r)$, we prove the theorem by induction on $s\geq 0$.
If $\mu_1<\lambda_1$, then $E=0$ by Lemma \ref{lemma:hojyo_1} (B). 
If $\mu_1>\lambda_1$, then $E=0$ by Lemma \ref{lemma:hojyo_1} (C).
Assume $\mu_1=\lambda_1$.
If $\mu_1=\lambda_1=0$, then we have
\[
\begin{aligned}
E
&=
\bra{0}e^{\theta}(\phi^{(\beta)}_0)^\ast \phi_0^{[\beta]}e^{-\theta}\ket{0}\\
&
=
\left\langle(\phi^{(\beta)}_0-\beta \phi^{(\beta)}_1+\beta^2\phi^{(\beta)}_2-\cdots )^\ast (\phi_0^{[\beta]}+\beta \phi_1^{[\beta]})
\right\rangle
\qquad (\text{Corollary~\ref{cor:sublemma}})
\\
&
=\left\langle
(\phi^{(\beta)}_0)^\ast \phi_0^{[\beta]}
\right\rangle\qquad (\text{Eq.~\eqref{eq:ann_rule}})\\
&=1\qquad (\text{Eq.~\eqref{eq:vs_phi_0}}).
\end{aligned}
\]
If $\mu_1=\lambda_1>0$, let 
$\lambda'=(\lambda_2>\dots>\lambda_r\geq 0)$ and
$\mu'=(\mu_2>\dots>\mu_s\geq 0)$.
Then, we have
\[
\begin{aligned}
E
=
(\mu|\lambda)
&=
\mbra{\mu'}
e^\theta (\Phi^{(\beta)}_{\mu_1})^\ast
\phi^{[\beta]}_{\lambda_1} e^{-\theta}
\mket{\lambda'}\\
&
=
\mbra{\mu'}
e^\theta 
\{
1-\phi^{[\beta]}_{\lambda_1} (\Phi^{(\beta)}_{\mu_1})^\ast
\}
e^{-\theta}
\mket{\lambda'}\qquad
(\mathrm{Lemma\ \ref{lemma:duality}})
\\
&=
(\mu'|\lambda')
-
\mbra{\mu'}
e^\theta 
\phi^{[\beta]}_{\lambda_1} (\Phi^{(\beta)}_{\mu_1})^\ast
e^{-\theta}
\mket{\lambda'}\\
&=
( \mu'|\lambda')
-
\mbra{\mu'}
e^\theta 
\phi^{[\beta]}_{\lambda_1} 
e^{-\theta}
(\Phi^{(\beta)}_{\mu_1}-\beta \Phi^{(\beta)}_{\mu_1+1}+\cdots)^\ast
\mket{\lambda'}\qquad (\mathrm{Corollary~\ref{cor:sublemma}}).
\end{aligned}
\]
The second term in the last expression is $0$ by Lemma \ref{lemma:hojyo_1} (C).
Hence, we have $E=(\mu'|\lambda').$
By induction hypothesis, we conclude that $E=\delta_{\mu,\lambda}$.
\end{proof}

\subsection{Generating function for $gq$}

As an application of the fermionic expression given in Theorem \ref{thm:gq}, we can derive a generating function of the $gq$-functions.
Let $gq_n=gq_{(n)}$ be the $gq$-function corresponding to the one-row partition $(n)$.
Let $gq(z)=\sum_{n=0}^\infty gq_n(x)z^n$ be the generating function of $gq_n$.

\begin{prop}
We have
\begin{equation}\label{eq:gen_gq_one_row}
gq(z)=\left\langle
e^{\mathcal{H}^{[\beta]}}\phi^{[\beta]}(z)\phi_0
\right\rangle
=\prod_i\frac{1-x_i\overline{z}}{1-x_iz}
\end{equation}
and
\begin{equation}\label{eq:generating_gq_gen}
e^{\mathcal{H}^{[\beta]}}\phi^{[\beta]}(z)e^{-\mathcal{H}^{[\beta]}}=gq(z)\cdot \phi^{[\beta]}(z).
\end{equation}
\end{prop}
\begin{proof}
This proposition immediately follows from Lemma \ref{lemma:commutation_rels} \eqref{item:2-4} and Theorem \ref{thm:gq}.
\end{proof}

We next derive a generating function and a Pfaffian formula for $gq_\lambda$ for general $\lambda$.
Let $r$ be an even integer.
By Theorem \ref{thm:gq}, the formal series
\[
\mathfrak{gq}(z_1,\dots,z_r)
:=\left\langle
e^{\mathcal{H}^{[\beta]}}
\phi^{[\beta]}(z_1)e^{-\theta}
\phi^{[\beta]}(z_2)e^{-\theta}
\dots
\phi^{[\beta]}(z_r)e^{-\theta} 
\right\rangle
\]
is a generating function of $gq_\lambda$ for $\ell(\lambda)\leq r$.
By Lemma~\ref{lemma:commutation_rels}, $\mathfrak{gq}(z_1,\dots,z_r)$ can be computed as follows:
\begin{align}
\mathfrak{gq}(z_1,\dots,z_r)\nonumber
&=
\prod_{i=1}^r\frac{1}{(1+\beta z_i)^{i-1}}
\left\langle
e^{\mathcal{H}^{[\beta]}}
\phi^{[\beta]}(z_1)
\phi^{[\beta]}(z_2)
\dots
\phi^{[\beta]}(z_r)
\right\rangle\nonumber
\qquad (\text{Lemma~\ref{lemma:commutation_rels}~\eqref{item:2-1},\eqref{item:2-2}})
\allowdisplaybreaks
\\
&=
\prod_{i=1}^r\frac{1}{(1+\beta z_i)^{i-1}}
\cdot \mathrm{Pf}\left(
\langle
e^{\mathcal{H}^{[\beta]}}
\phi^{[\beta]}(z_i)
\phi^{[\beta]}(z_j)
\rangle
\right)_{1\leq i<j\leq r}\nonumber
\qquad (\mathrm{Eq.~\eqref{eq:Wick}})
\allowdisplaybreaks
\\
&=
\prod_{i=1}^r\frac{1}{(1+\beta z_i)^{i-1}}
\cdot
\mathrm{Pf}\left(
gq(z_i)gq(z_j)
\langle
\phi^{[\beta]}(z_i)
\phi^{[\beta]}(z_j)
\rangle
\right)_{1\leq i<j\leq r}\nonumber
\qquad (\mathrm{Eq.~\eqref{eq:generating_gq_gen}})
\allowdisplaybreaks
\\
&=
\mathrm{Pf}\left(
\frac{gq(z_i)}{(1+\beta z_i)^{i-1}}
\frac{gq(z_j)}{(1+\beta z_j)^{j-1}}
\cdot 
\frac{z_i-z_j}{z_i\oplus z_j}
\right)_{1\leq i<j\leq r}.\label{eq:Pf_cont}
\end{align}
By \eqref{eq:IN-formula}, the Pfaffian \eqref{eq:Pf_cont} is rewritten as
\[
\prod_{i=1}^r\frac{gq(z_i)}{(1+\beta z_i)^{i-1}}
\cdot
\prod_{1\leq i<j\leq r}
\frac{z_i-z_j}{z_i\oplus z_j}=
\prod_{i=1}^r gq(z_i)
\cdot
\prod_{1\leq i<j\leq r}
\frac{z_i\ominus z_j}{z_i\oplus z_j}.
\]
Thus, we obtain
\begin{equation}\label{eq:gq_gen}
gq_\lambda=[z_1^{\lambda_1}\dots z_r^{\lambda_r}]
\prod_{i=1}^r gq(z_i)
\cdot
\prod_{1\leq i<j\leq r}
\frac{z_i\ominus z_j}{z_i\oplus z_j},
\end{equation}
where the rational function on the right hand side is regarded as an element of \[
\QQ(\beta)[x_1,x_2,\dots]((z_1))\cdots ((z_r)).
\]

The generating function \eqref{eq:gq_gen} was conjectured in \cite[Conjecture 5.3]{nakagawa2023universal}.
Since \eqref{eq:gq_gen} admits a substitution $\lambda_r=0$, we immediately obtain $gq_{(\lambda,0)}=gq_{\lambda}$.

\begin{cor}
Let $\lambda$ be a strict partition, and $r$ be the smallest even integer greater than or equal to $\ell(\lambda)$.
Then, the $gq$-function admits the following Pfaffian formula:
\[
gq_\lambda(x)=\mathrm{Pf}\left(
\sum_{v=0}^{\lambda_j}
\sum_{u=-v}^{\lambda_i} c^{(i,j)}_{u,v}gq_{\lambda_i-u}(x)\cdot gq_{\lambda_j-v}(x)
\right)_{1\leq i<j\leq r},
\]
where
\[
\frac{1}{(1+\beta t)^{i-1}}\frac{1}{(1+\beta s)^{j-1}}\frac{t- s}{t\oplus s}
=
\sum_{v=0}^\infty 
\sum_{u=-v}^\infty 
c^{(i,j)}_{u,v}t^us^v
\]
be the formal expansion in the field $\QQ(\beta)((t))((s))$.
\end{cor}
\begin{proof}
When expanding the Pfaffian \eqref{eq:Pf_cont} using the definition \eqref{eq:Pfaffian_def}, each variable $z_i$ appears exactly once in every term.
Therefore, we have 
\begin{align*}
gq_\lambda(x)&=[z_1^{\lambda_1}\dots z_r^{\lambda_r}]
\mathrm{Pf}\left(
\frac{gq(z_i)}{(1+\beta z_i)^{i-1}}
\frac{gq(z_j)}{(1+\beta z_j)^{j-1}}
\cdot 
\frac{z_i-z_j}{z_i\oplus z_j}
\right)_{1\leq i<j\leq r}\allowdisplaybreaks\\
&=
\mathrm{Pf}\left([z_1^{\lambda_1}\dots z_r^{\lambda_r}]
\frac{gq(z_i)}{(1+\beta z_i)^{i-1}}
\frac{gq(z_j)}{(1+\beta z_j)^{j-1}}
\cdot 
\frac{z_i-z_j}{z_i\oplus z_j}
\right)_{1\leq i<j\leq r},
\end{align*}
which completes the proof.
\end{proof}

\subsection{Fermionic description of $gp$}
\label{sec:fermionic_pre_of_gp}

One might expect that the $gp$-functions also admit an expression similar to that of the $gq$-functions.
However, the fermionic expression of $gp_\lambda$ turns out to significantly more complicated.
This is due to the fact that, while the vector $\bra{0}\phi_n^{[\beta]}$ vanishes for $n>0$, $\bra{0}\Phi_n^{[\beta]}$ does not, as shown in \eqref{eq:lack_of_ann_rule}.

Let 
\begin{equation}\label{eq:gp_prime}
gp'_\lambda=
\begin{cases}
\langle
e^{\mathcal{H}^{[\beta]}}
\Phi^{[\beta]}_{\lambda_1}e^{-\theta}
\Phi^{[\beta]}_{\lambda_2}e^{-\theta}
\dots
\Phi^{[\beta]}_{\lambda_r}e^{-\theta}
\rangle & (r:\mbox{even}),\\
\langle
e^{\mathcal{H}^{[\beta]}}
\Phi^{[\beta]}_{\lambda_1}e^{-\theta}
\Phi^{[\beta]}_{\lambda_2}e^{-\theta}
\dots
\Phi^{[\beta]}_{\lambda_r}e^{-\theta}
\phi^{[\beta]}_{0}e^{-\theta}
\rangle & (r:\mbox{odd}).
\end{cases}
\end{equation}
By analogy with Theorem \ref{thm:gq}, one might expect that $gp'_\lambda$ coincides with $gp_\lambda$.
However, this is not the case.
In fact, we have
\begin{equation}\label{eq:Fal}
\langle
GQ_{\lambda},gp'_\mu
\rangle\neq \delta_{\lambda,\mu}
\end{equation}
if $\lambda\neq \emptyset$ or $\mu\neq \emptyset$.
For example, when $\lambda=\emptyset$ and $\mu=(n)$, we have
\[
\langle
GQ_\emptyset,gp_n'
\rangle
=
\left\langle
e^\theta (\phi_0^{(\beta)})^\ast
\Phi^{[\beta]}_ne^{-\theta}
\right\rangle
=
\frac{(-\beta)^n}{2^{n+1}}
\left\langle
e^\theta (\phi_0^{(\beta)})^\ast
\phi^{[\beta]}_0e^{-\theta}
\right\rangle
=
\frac{(-\beta)^n}{2^{n+1}}
\neq 0
\]
as follows from \eqref{eq:lack_of_ann_rule}.

In order to obtain a correct expression for $gp_\lambda$, we introduce a new vector $\ket{\lambda}^+_p\in \mathcal{F}$ defined by
\begin{equation}\label{eq:ket_p}
\ket{\lambda}^+_p
=
\left(\Phi^{[\beta]}_{\lambda_1}-\tfrac{1}{2}(-\tfrac{\beta}{2})^{\lambda_1}\right)e^{-\theta}
\left(\Phi^{[\beta]}_{\lambda_2}-\tfrac{1}{2}(-\tfrac{\beta}{2})^{\lambda_2}\right)e^{-\theta}
\cdots
e^{-\theta}
\left(\Phi^{[\beta]}_{\lambda_r}-\tfrac{1}{2}(-\tfrac{\beta}{2})^{\lambda_r}\right)(\phi_0+1)\ket{0}
\end{equation}
for a strict partition $\lambda=(\lambda_1>\dots>\lambda_r>0)$.
Note that $\ket{\lambda}^+_p$ is not contained in $\mathcal{F}_{even}$.

Recall that the Fock space $\mathcal{F}$ is uniquely decomposed as  $\mathcal{F}=\mathcal{F}_{odd}\oplus \mathcal{F}_{even}$ into the odd and the even subspaces.
Let $\pi_{e}:\mathcal{F}\to \mathcal{F}_{even}$ be the projection to the even part along to the decomposition.
Then, we define the vector $\ket{\lambda}_p\in \mathcal{F}_{even}$ by 
\[
\ket{\lambda}_p:=\pi_e\left(\ket{\lambda}^+_p\right).
\]

\begin{thm}\label{thm:fermion_of_gp}
For a strict partition $\lambda=(\lambda_1>\dots>\lambda_r>0)$, we have
\begin{equation}\label{eq:def_of_gp}
\begin{aligned}
gp_\lambda(x)=\bra{0}e^{\mathcal{H}^{[\beta]}}\ket{\lambda}_p.
\end{aligned}
\end{equation}
\end{thm}

Recall that, when $\bra{v}\in \mathcal{F}^\ast_{even}$ and $\ket{w}\in \mathcal{F}_{odd}$, the expectation value $\inner{v}{w}$ vanishes automatically.
Then, we have
$
\langle v|\lambda\rangle_p=\langle v|\lambda\rangle_p^+
$
for all $\bra{v}\in \mathcal{F}^\ast_{even}$.
Hence, in order to prove Theorem \ref{thm:fermion_of_gp}, it suffices to show
\begin{equation}\label{eq:to_prove}
{}_{Q}\langle
\kappa|\lambda
\rangle^+_p
=\delta_{\lambda,\kappa}\qquad \text{for all strict partitions $\kappa$},
\end{equation}
which is equivalent to ${}_{Q}\langle
\kappa|\lambda
\rangle_p
=\delta_{\lambda,\kappa}$.

For a strict partition $\lambda=(\lambda_1>\dots>\lambda_r>0)$ and a decreasing sequence
$\mu=(\mu_1>\dots>\mu_s\geq 0)$,
we introduce two auxiliary vectors $\bbra{\mu}$ and $\kket{\lambda}$ defined by
\[
\begin{aligned}
&
\bbra{\mu}=\bra{0}(\phi_0+1)
e^\theta (\phi^{(\beta)}_{\mu_s})^\ast
\dots
e^\theta (\phi^{(\beta)}_{\mu_2})^\ast
e^\theta (\phi^{(\beta)}_{\mu_1})^\ast,\\
&
\kket{\lambda}=
\left(\Phi^{[\beta]}_{\lambda_1}-\tfrac{1}{2}(-\tfrac{\beta}{2})^{\lambda_1}\right)e^{-\theta}
\left(\Phi^{[\beta]}_{\lambda_2}-\tfrac{1}{2}(-\tfrac{\beta}{2})^{\lambda_2}\right)e^{-\theta}
\cdots
\left(\Phi^{[\beta]}_{\lambda_r}-\tfrac{1}{2}(-\tfrac{\beta}{2})^{\lambda_r}\right)\ket{0}.
\end{aligned}
\]
Since
\[
\left\langle (\phi_0+1)X\right\rangle
=
\left\langle X(\phi_0+1)\right\rangle,\qquad
\text{for all }X\in \mathcal{A}
\]
and
\[
{}_Q\bra{\kappa}=
(\ket{\kappa}_Q)^\ast=
\begin{cases}
\bra{0}e^\theta (\phi^{(\beta)}_{\kappa_t})^\ast
\dots
e^\theta (\phi^{(\beta)}_{\kappa_2})^\ast
e^\theta (\phi^{(\beta)}_{\kappa_1})^\ast& (\text{if }t=\ell(\kappa)\mbox{ is even}),\\
\bra{0}e^\theta (\phi^{(\beta)}_{0})^\ast e^\theta (\phi^{(\beta)}_{\kappa_t})^\ast
\dots
e^\theta (\phi^{(\beta)}_{\kappa_2})^\ast
e^\theta (\phi^{(\beta)}_{\kappa_1})^\ast  & (\text{if }t=\ell(\kappa)\mbox{ is odd}),
\end{cases}
\]
which is obtained from \eqref{eq:vector_Q}, we have
\[
{}_{Q}\langle
\kappa|\lambda
\rangle^+_p=
\begin{cases}
\iinner{\kappa}{\lambda} & (\ell(\kappa)\text{ is even} ),\\
\iinner{(\kappa,0)}{\lambda} & (\ell(\kappa)\text{ is odd} ).
\end{cases}
\]
Hence, to show \eqref{eq:to_prove}, it suffices to verity
\begin{equation}\label{eq:mokuhyou}
\iinner{\mu}{\lambda}=
\begin{cases}
1 & (\mu=\lambda \text{ or } \mu=(\lambda,0)),\\
0 & (\text{otherwise}).
\end{cases}
\end{equation}

\begin{lemma}\label{lemma:hojyo2}
We have the following relations:
\begin{enumerate}
\item[(A).] When $\mu=\emptyset$ and $n\geq 0$, we have $\bbra{\emptyset}(\Phi^{[\beta]}_n-\tfrac{1}{2}(-\tfrac{\beta}{2})^{n})=0$.
\item[(B).] When $\mu\neq \emptyset$ and $n> \mu_1$, we have $\bbra{\mu}(\Phi^{[\beta]}_n-\tfrac{1}{2}(-\tfrac{\beta}{2})^{n})=0$.
\item[(C).] When $\lambda\neq \emptyset$ and $m>\lambda_1$, we have $(\phi^{(\beta)}_m)^\ast \kket{\lambda}=0$.
\end{enumerate}
\end{lemma}

\begin{proof}
(A):
A direct calculation shows that
\[
\begin{aligned}
\bbra{\emptyset}(\Phi^{[\beta]}_n-\tfrac{1}{2}(-\tfrac{\beta}{2})^{n})
&=
\bra{0}(\phi_0+1)(\Phi_n^{[\beta]}-\tfrac{1}{2}(-\tfrac{\beta}{2})^n)\\
&=
\tfrac{1}{2}(-\tfrac{\beta}{2})^n
\bra{0}(\phi_0+1)(\phi_0-1)\qquad (\mathrm{Eq.~\eqref{eq:Phi_expand}})\\
&=
\tfrac{1}{2}(-\tfrac{\beta}{2})^n
\bra{0}(\phi_0^2-1)\\
&=0 \qquad (\phi_0^2=1).
\end{aligned}
\]

(B):
Let $\mu=(\mu_1>\dots>\mu_s\geq 0)$ and $\mu'=(\mu_2>\dots>\mu_s\geq 0)$.
We show (B) by induction on $s\geq 1$. By assumption, we have $n>0$ and
\begin{align*}
\bbra{\mu}(\Phi^{[\beta]}_n-\tfrac{1}{2}(-\tfrac{\beta}{2})^{n})
&=
\bbra{\mu'}e^{\theta}(\phi^{(\beta)}_{\mu_1})^\ast(\Phi^{[\beta]}_n-\tfrac{1}{2}(-\tfrac{\beta}{2})^{n})\\
&=
\bbra{\mu'}e^{\theta}(-\Phi^{[\beta]}_n-\tfrac{1}{2}(-\tfrac{\beta}{2})^{n})(\phi^{(\beta)}_{\mu_1})^\ast\qquad (\text{Lemma~\ref{lemma:duality}})\\
&=
\bbra{\mu'}(-\Phi^{[\beta]}_n-\beta\Phi^{[\beta]}_{n-1} -\tfrac{1}{2}(-\tfrac{\beta}{2})^{n})e^{\theta}(\phi^{(\beta)}_{\mu_1})^\ast \qquad (\text{Corollary~\ref{cor:sublemma}}).
\end{align*}
Since
\[
-\Phi^{[\beta]}_n-\beta\Phi^{[\beta]}_{n-1} -\tfrac{1}{2}(-\tfrac{\beta}{2})^{n}=
-(\Phi^{[\beta]}_n-\tfrac{1}{2}(-\tfrac{\beta}{2})^n)
-\beta (\Phi^{[\beta]}_{n-1}-\tfrac{1}{2}(-\tfrac{\beta}{2})^{n-1})
,
\]
we see that the vector 
$
\bbra{\mu}(\Phi^{[\beta]}_n-\tfrac{1}{2}(-\tfrac{\beta}{2})^{n})
$
is a linear combination of 
\[
\bbra{\mu'}(\Phi^{[\beta]}_n-\tfrac{1}{2}(-\tfrac{\beta}{2})^{n})e^{\theta}(\phi^{(\beta)}_{\mu_1})^\ast
\qquad \text{and}\qquad  
\bbra{\mu'}(\Phi^{[\beta]}_{n-1}-\tfrac{1}{2}(-\tfrac{\beta}{2})^{n-1})e^{\theta}(\phi^{(\beta)}_{\mu_1})^\ast.
\]
When $s=1$, the desired identity reduces to (A).
For general $s>1$, (B) follows from the induction hypothesis.

(C):
This claim can be proved in a manner similar to that of the proof of Lemma \ref{lemma:hojyo_1} (C).
\end{proof}

\begin{proof}[Proof of Theorem \ref{thm:fermion_of_gp}]

Let $E=\iinner{\mu}{\lambda}$.
To prove the theorem, it suffices to verify \eqref{eq:mokuhyou}.
We show \eqref{eq:mokuhyou} by induction on $s\geq 0$.
We begin by considering the following base cases:
\begin{equation}\label{eq:base_case}
\iinner{\mu}{\emptyset}=
\begin{cases}
1 & (\mu=\emptyset),\\
1 & (\mu=(0)),\\
0 & (\text{otherwise}),\\
\end{cases}\qquad
\iinner{\emptyset}{\lambda}
=
\begin{cases}
1 & (\lambda=\emptyset),\\
0 & (\lambda\neq \emptyset).
\end{cases}
\end{equation}

The first equation in \eqref{eq:base_case} can be verified as follows: When $\mu=\emptyset$, we have $E=\bra{0}(\phi_0+1)\ket{0}=1$.
When $\mu=(0)$, we have \[
E=
\bra{0}(\phi_0+1)e^\theta(\phi_0^{(\beta)})^\ast\ket{0}
=\bra{0}(\phi_0+1)e^\theta\phi_0\ket{0}
=\bra{0}(\phi_0^2+\phi_0)\ket{0}=1.
\]
When $\mu\neq \emptyset$ and $\mu\neq(0)$, we have $E=0$ from Lemma \ref{lemma:hojyo2} (C).

The second equation in \eqref{eq:base_case} can be verified as follows: We have already shown $E=1$ when $\lambda=\emptyset$.
When $\lambda\neq \emptyset$, we have $E=0$ from Lemma \ref{lemma:hojyo2} (A).

We proceed for the case when $s,r>0$.
If $\mu_1<\lambda_1$, then $E=0$ by Lemma \ref{lemma:hojyo2} (B).
If $\mu_1>\lambda_1$, then $E=0$ by Lemma \ref{lemma:hojyo2} (C).
If $\mu_1=\lambda_1(>0)$, then we have
\[
\begin{aligned}
E
&=
\bbra{\mu'}e^{\theta}(\phi^{(\beta)}_{\mu_1})^\ast
(\Phi^{[\beta]}_{\lambda_1}-\tfrac{1}{2}(-\tfrac{\beta}{2})^{\lambda_1})e^{-\theta}\kket{\lambda'}\\
&=
\bbra{\mu'}e^{\theta}
\left[
(\phi^{(\beta)}_{\mu_1})^\ast,
\Phi^{[\beta]}_{\lambda_1}
\right]_+
e^{-\theta}\kket{\lambda'}
-
\bbra{\mu'}e^{\theta}
\Phi^{[\beta]}_{\lambda_1}
(\phi^{(\beta)}_{\mu_1})^\ast
e^{-\theta}\kket{\lambda'}
-\tfrac{1}{2}(-\tfrac{\beta}{2})^{\lambda_1}
\bbra{\mu'}e^{\theta}
(\phi^{(\beta)}_{\mu_1})^\ast
e^{-\theta}\kket{\lambda'}\\
&
=
\iinner{\mu'}{\lambda'}
-
\bbra{\mu'}e^{\theta}
\Phi^{[\beta]}_{\lambda_1}
(\phi^{(\beta)}_{\mu_1})^\ast
e^{-\theta}\kket{\lambda'}
-\tfrac{1}{2}(-\tfrac{\beta}{2})^{\lambda_1}
\bbra{\mu'}e^{\theta}
(\phi^{(\beta)}_{\mu_1})^\ast
e^{-\theta}\kket{\lambda'}
\qquad 
(\mathrm{Lemma\,\ref{lemma:duality}}).
\end{aligned}
\]
The last two terms in the last expression equal to $0$ because
\[
(\phi^{(\beta)}_{\mu_1})^\ast
e^{-\theta}\kket{\lambda'}=
e^{-\theta}
(\phi^{(\beta)}_{\mu_1}-\beta \phi^{(\beta)}_{\mu_1+1}+\cdots)^\ast
\kket{\lambda'}
=
0
\]
by Lemma~\ref{lemma:hojyo2}\,(C).
Hence, we obtain $E=\iinner{\mu'}{\lambda'}$, where
$\mu'=(\mu_2>\dots>\mu_s\geq 0)$ and $\lambda'=(\lambda_2>\dots>\lambda_r>0)$.
By induction hypothesis, we conclude \eqref{eq:mokuhyou}.
\end{proof}

\subsection{Generating function of $gp_\lambda$}

Let $gp_n=gp_{(n)}$ be the $gp$-function corresponding to the one-row partition $(n)$.
From Theorem \ref{thm:fermion_of_gp}, we have
\begin{align*}
\sum_{n=1}^\infty gp_nz^n
&=\sum_{n=1}^\infty \bra{0}e^{\mathcal{H}^{[\beta]}}\ket{(n)}_p\cdot z^n
=\sum_{n=1}^\infty
\bra{0}e^{\mathcal{H}^{[\beta]}}\ket{(n)}^+_p\cdot z^n\\
&=\sum_{n=1}^\infty\bra{0}e^{\mathcal{H}^{[\beta]}}\left(\Phi^{[\beta]}_{n}-\tfrac{1}{2}(-\tfrac{\beta}{2})^{n}\right)
(\phi_0+1)\ket{0}\cdot z^n\allowdisplaybreaks\\
&=
\bra{0}e^{\mathcal{H}^{[\beta]}}\left(\Phi^{[\beta]}(z)-\frac{1}{2+\beta z}\right)
(\phi_0+1)\ket{0}
-
\bra{0}e^{\mathcal{H}^{[\beta]}}\left(\Phi^{[\beta]}_0-\frac{1}{2}\right)
(\phi_0+1)\ket{0}.
\end{align*}
From \eqref{eq:ann_Phi_0}, the second term of the last expression equals to
\[
\bra{0}e^{\mathcal{H}^{[\beta]}}\left(\frac{1}{2}\phi_0-\frac{1}{2}\right)
(\phi_0+1)\ket{0}
=
\frac{1}{2}
\bra{0}(\phi_0^2-1)\ket{0}=0.
\]
Therefore, we have
\begin{align*}
\sum_{n=1}^\infty gp_nz^n
&=
\bra{0}e^{\mathcal{H}^{[\beta]}}\left(\Phi^{[\beta]}(z)-\frac{1}{2+\beta z}\right)
(\phi_0+1)\ket{0}\allowdisplaybreaks\\
&=\frac{1}{2+\beta z}
\bra{0}e^{\mathcal{H}^{[\beta]}}
\left(
\phi^{[\beta]}(z)-1
\right)(\phi_0+1)\ket{0}\allowdisplaybreaks\\
&=\frac{1}{2+\beta z}(gq(z)-1).
\end{align*}

For general $\lambda$, it follows from Theorem \ref{thm:fermion_of_gp} that
\begin{equation}\label{eq:gp_generating}
gp_\lambda=[z_1^{\lambda_1}\dots z_r^{\lambda_r}]
\left\langle
e^{\mathcal{H}^{[\beta]}}
\left(
\Phi^{[\beta]}(z_1)-\frac{1}{2+\beta z_1}
\right)
e^{-\theta}
\cdots
e^{-\theta}
\left(
\Phi^{[\beta]}(z_r)-\frac{1}{2+\beta z_r}
\right)
(\phi_0+1)\allowdisplaybreaks
\right\rangle.
\end{equation}
Calculating the vacuum expectation value on the right hand side of \eqref{eq:gp_generating} is quite involved, because the operator inside the brackets mixes odd and even operators. 
Apparently, the only possible approach is to expand it into $2^r$ terms as follows:
\[
\begin{aligned}
&\prod_{i=1}^r\frac{1}{2+\beta z_i}
\left\langle
e^{\mathcal{H}^{[\beta]}}
\left(
\phi^{[\beta]}(z_1)-1
\right)
e^{-\theta}
\cdots
e^{-\theta}
\left(
\phi^{[\beta]}(z_r)-1
\right)
(\phi_0+1)
\right\rangle\\
&=
\prod_{i=1}^r\frac{1}{2+\beta z_i}
\sum_{
\substack{
0\leq a\leq r\\
i_1<i_2<\dots<i_a
}
}
(-1)^{r-a}I_{i_1,\dots,i_a}
\end{aligned}
\]
where 
\[
I_{i_1,\dots,i_a}:=\left\langle
e^{\mathcal{H}^{[\beta]}}
e^{-(i_1-1)\theta}
\phi^{[\beta]}(z_{i_1})
e^{-(i_2-i_1)\theta}
\phi^{[\beta]}(z_{i_2})
\cdots
e^{-(i_a-i_{a-1})\theta}
\phi^{[\beta]}(z_{i_a})
(\phi_0+1)
\right\rangle.
\]
By similar calculations to those used to derive \eqref{eq:gq_gen}, the vacuum expectation value $I_{i_1,\dots,i_a}$ can be computed as follows:
\begin{align}
I_{i_1,\dots,i_a}
&=\prod_{\kappa = 1}^a
\frac{1}{(1+\beta z_{i_\kappa})^{i_\kappa-1}}
\left\langle
e^{\mathcal{H}^{[\beta]}}
\phi^{[\beta]}(z_{i_1})
\cdots
\phi^{[\beta]}(z_{i_a})
(\phi_0+1)
\right\rangle\qquad (\text{Lemma \ref{lemma:commutation_rels} \eqref{item:2-1}, \eqref{item:2-2}})
\label{eq:derivation_of_I}
\\
&=\prod_{\kappa = 1}^a
\frac{gq(z_{i_\kappa})}{(1+\beta z_{i_\kappa})^{i_\kappa-1}}
\prod_{b<c}\frac{z_{i_b}-z_{i_c}}{z_{i_b}\oplus z_{i_c}}
\nonumber\\
&=\prod_{\kappa = 1}^a
\frac{gq(z_{i_\kappa})}{(1+\beta z_{i_\kappa})^{i_\kappa-\kappa}}
\prod_{b<c}\frac{z_{i_b}\ominus z_{i_c}}{z_{i_b}\oplus z_{i_c}}.\nonumber
\end{align}

\begin{prop}\label{prop:gen_of_gp}
The $gp$-function satisfies the following equation:
\[
gp_\lambda=
[z_1^{\lambda_1},\dots,z_r^{\lambda_r}]
\left(
\prod_{i=1}^r\frac{1}{2+\beta z_i}
\sum_{
\substack{
0\leq a\leq r\\
i_1<i_2<\dots<i_a
}
}
(-1)^{r-a}I_{i_1,\dots,i_a}
\right),
\]
where $I_{i_1,\dots,i_a}$ is the rational function given in \eqref{eq:derivation_of_I}, expanded in the field $\QQ(\beta)((z_r))\dots ((z_2))((z_1))$.
\end{prop}

\bibliographystyle{plain}
\bibliography{Groth}

\appendix

\section{Proof of \eqref{eq:substitution_GQ}}

Let
\[
F(u_1,\dots,u_r):=
\prod_{i=1}^r
\frac{1}{1+\beta u_i}
\prod_{i,j}\frac{u_i\oplus x_j}{u_i\ominus x_j}
\prod_{1\leq i<j\leq r}\frac{u_j\ominus u_i}{u_j\oplus u_i}
\]
be the rational function on the right hand side of \eqref{eq:gen_of_GQ}.
The rational function $F(u_1,\dots,u_r)$ is understood as an element of the field $\QQ(\beta)((u_r))\dots ((u_2))((u_1))[[x_1,x_2,\dots]]
$ via the Laurent expansion on the domain $D_{(r)}:=\{|x_j|<|u_1|<|u_2|<\dots<|u_r|<|\beta^{-1}|:\forall j\}$.

Putting $r\mapsto r+1$ and substituting $\lambda_{r+1}=0$ to \eqref{eq:gen_of_GQ}, we have
\[
\begin{aligned}
GQ_{(\lambda,0)}
&
=[u_1^{-\lambda_1}\dots u_{r}^{-\lambda_r}u_{r+1}^0]
\prod_{i=1}^{r+1}\frac{1}{1+\beta u_i}
\prod_{i,j}\frac{u_i\oplus x_j}{u_i\ominus x_j}
\prod_{1\leq i<j\leq r+1}\frac{u_j\ominus u_i}{u_j\oplus u_i}\\
&=
[u_1^{-\lambda_1}\dots u_{r}^{-\lambda_r}]
\prod_{i=1}^{r}\frac{1}{1+\beta u_i}
\prod_{i,j}\frac{u_i\oplus x_j}{u_i\ominus x_j}
\prod_{1\leq i<j\leq r}\frac{u_j\ominus u_i}{u_j\oplus u_i}\\
&\hspace{5em}\times
\frac{1}{2\pi i}
\oint_{u_{r+1}\in D_{(r+1)}}
\frac{1}{1+\beta u_{r+1}}
\prod_j\frac{u_{r+1}\oplus x_j}{u_{r+1}\ominus x_j}
\prod_{i=1}^r\frac{u_{r+1}\ominus u_i}{u_{r+1}\oplus u_i}
\frac{du_{r+1}}{u_{r+1}}\\
&
=[u_1^{-\lambda_1}\dots u_{r}^{-\lambda_r}]
\prod_{i=1}^{r}\frac{1}{1+\beta u_i}
\prod_{i,j}\frac{u_i\oplus x_j}{u_i\ominus x_j}
\prod_{1\leq i<j\leq r}\frac{u_j\ominus u_i}{u_j\oplus u_i}\\
&\hspace{5em}
\times
\mathop{\mathrm{Res}}_{w=-\beta}
\left(
-\frac{1}{1+\beta w^{-1}}
\prod_j\frac{w^{-1}\oplus x_j}{w^{-1}\ominus x_j}
\prod_{i=1}^r\frac{w^{-1}\ominus u_i}{w^{-1}\oplus u_i}
\frac{dw}{w}
\right).\qquad (w^{-1}=u_{r+1})
\end{aligned}
\]
A straightforward calculation shows that the residue in the last expression equals $1$.
Therefore, we have $GQ_{(\lambda,0)}=GQ_\lambda$.

\end{document}